\documentclass[11pt]{article}

\usepackage{amsmath,amssymb,amsthm}
\usepackage{amscd}
\usepackage{bm}
\usepackage{graphicx}
\usepackage{ascmac}
\usepackage{BOONDOX-frak, mathrsfs}
\usepackage[top=30truemm,bottom=30truemm,left=25truemm,right=25truemm]{geometry}
\usepackage[nottoc,notlot,notlof]{tocbibind} 

\newcommand{\Z}{\mathbb{Z}}
\newcommand{\Q}{\mathbb{Q}}
\newcommand{\F}{\mathbb{F}}
\newcommand{\pe}{\mathfrak{p}}
\newcommand{\Pe}{\mathfrak{P}}
\newcommand{\tilk}{\tilde{k}}

\newcommand{\KKi}{K^{(i)}}
\newcommand{\KKip}{K^{(i+1)}}
\newcommand{\LL}{\mathscr{L}}
\newcommand{\II}{\mathcal{I}}
\newcommand{\MM}{\mathscr{M}}
\newcommand{\FF}{\mathscr{F}}

\newcommand{\BB}{\mathcal{B}}
\newcommand{\CC}{\mathcal{C}}
\newcommand{\EE}{\mathcal{E}}

\DeclareMathOperator{\Gal}{Gal}
\DeclareMathOperator{\Image}{Im}
\DeclareMathOperator{\cyc}{cyc}
\DeclareMathOperator{\Ann}{Ann}
\DeclareMathOperator{\GL}{GL}
\DeclareMathOperator{\Aut}{Aut}
\DeclareMathOperator{\End}{End}

\DeclareMathOperator{\Gr}{Gr}
\DeclareMathOperator{\Ker}{Ker}
\DeclareMathOperator{\rank}{rank}

\DeclareMathOperator{\height}{ht}
\DeclareMathOperator{\ch}{char}
\DeclareMathOperator{\ram}{ram}
\DeclareMathOperator{\sfin}{sf}
\DeclareMathOperator{\ns}{ns}

\let\oldenumerate\enumerate
\renewcommand{\enumerate}{
   \oldenumerate
   \setlength{\itemsep}{1pt}
   \setlength{\parskip}{0pt}
   \setlength{\parsep}{0pt}
}
\let\olditemize\itemize
\renewcommand{\itemize}{
   \olditemize
   \setlength{\itemsep}{1pt}
   \setlength{\parskip}{0pt}
   \setlength{\parsep}{0pt}
}

\theoremstyle{plain}
\newtheorem{thm}{Theorem}[section]
\newtheorem{lem}[thm]{Lemma}
\newtheorem{conj}[thm]{Conjecture}
\newtheorem{claim}[thm]{Claim}
\newtheorem{prop}[thm]{Proposition}
\newtheorem{cor}[thm]{Corollary}
\theoremstyle{definition}
\newtheorem{defn}[thm]{Definition}
\newtheorem{rem}[thm]{Remark}
\newtheorem{eg}[thm]{Example}
\newtheorem{ass}[thm]{Assumption}

\title{A consequence of Greenberg's generalized conjecture on Iwasawa invariants of $\Z_p$-extensions}
\author{Takenori Kataoka
\thanks{Email: tkataoka@ms.u-tokyo.ac.jp
\newline Graduate school of Mathematical Sciences, the University of Tokyo}}
\date{}

\begin{document}
\maketitle

\begin{abstract}
For a prime number $p$ and a number field $k$,
let $\tilk$ be the compositum of all $\Z_p$-extensions of $k$.
Greenberg's Generalized Conjecture (GGC) claims the pseudo-nullity of the unramified Iwasawa module $X(\tilk)$ of $\tilk$.
It is known that, when $k$ is an imaginary quadratic field,
GGC has a consequence on the Iwasawa invariants associated to $\Z_p$-extensions of $k$.
In this paper, we partially generalize it to arbitrary number fields $k$.
\end{abstract}



\section{Introduction}\label{sec_introduction}

Let $p$ be a fixed prime number.
We fix an algebraic closure of the field $\Q$ of rational numbers
and any algebraic extension of $\Q$ is considered to be contained in it.

First we introduce some general notions in Iwasawa theory.
For any algebraic extension $F$ of $\Q$, 
let $L(F)$ be the maximal unramified pro-$p$ abelian extension of $F$
and let $X(F)$ be the Galois group $\Gal(L(F)/F)$.
When $k$ is a number field (i.e. a finite extension of $\Q$),
it is known by class field theory that $X(k)$ is canonically isomorphic to 
the $p$-Sylow subgroup of the ideal class group of $k$.
The structure of $X(F)$ is one of the main objects of study in number theory.

Let $k$ be a number field and $d$ a positive integer.
When $K/k$ is a $\Z_p^d$-extension,
let $\Lambda(K/k)$ be the completed group ring $\Z_p[[\Gal(K/k)]]$, which is often called the Iwasawa algebra.
It is known that $\Lambda(K/k)$ is non-canonically isomorphic to the ring of formal power series $\Z_p[[T_1, \dots, T_d]]$ and, in particular, $\Lambda(K/k)$ is a regular local ring.
In fact, if $\sigma_1 , \dots, \sigma_d$ constitute a $\Z_p$-basis of $\Gal(K/k)$, 
then an isomorphism $\Lambda(K/k) \overset{\sim}{\to} \Z_p[[T_1, \dots, T_d]]$ is obtained by sending $\sigma_i$ to $1+T_i$.
Since $L(K)/k$ is a Galois extension, we have the natural action of $\Gal(K/k)$ on $X(K)$ via the inner automorphisms.
This action defines the natural $\Lambda(K/k)$-module structure on $X(K)$.
It is known that $X(K)$ is a finitely generated torsion $\Lambda(K/k)$-module.
(See \cite{green2}. 
Although the statement there is the case where $K = \tilk$ defined below, one can modify the proof to arbitrary multiple $\Z_p$-extensions.)

In particular, for any number field $k$, let $\tilk$ be the compositum of all $\Z_p$-extensions of $k$.
It is known that $\tilk / k$ is a $\Z_p^{r_2(k) + 1 +  \delta(k,p)}$-extension,
where $r_2(k)$ is the number of complex places of $k$ and $\delta(k,p)$ is the Leopoldt's defect of $(k,p)$
(see \cite[Proposition (10.3.20)]{NSW}).
We put $d(k) = r_2(k) + 1 +  \delta(k, p)$, so $\tilk/k$ is a $\Z_p^{d(k)}$-extension.

We also need some ring theoretic materials 
\cite[Chapter V, \S 1]{NSW}.
In general, let $\Lambda$ be a noetherian integrally closed domain and $X$ a $\Lambda$-module. 
We say that $X$ is a {\it pseudo-null} $\Lambda$-module and write $X \sim 0$ or more precisely $X \sim_{\Lambda} 0$
if $X$ is finitely generated and 
the height of the annihilator ideal of $X$ is greater than or equal to $2$.
A homomorphism $X \to Y$ of $\Lambda$-modules is said to be a pseudo-isomorphism if its kernel and cokernel are both pseudo-null.
If there exists a pseudo-isomorphism $X \to Y$, we write $X \sim Y$ or $X \sim_{\Lambda} Y$.

Now Greenberg's Generalized Conjecture (GGC) claims the following.

\begin{conj}[{\cite[Conjecture 3.5]{green1}}]\label{conj_greenberg}
For any number field $k$, $X(\tilk)$ is pseudo-null as a $\Lambda(\tilk/k)$-module.
\end{conj}

We say that GGC holds for $(k,p)$ if $X(\tilk)$ is pseudo-null as a $\Lambda(\tilk/k)$-module.
Although GGC is still an open problem, there are some cases where GGC is known to be true.
For example, GGC holds for $(k, p)$ if $k$ is an imaginary quadratic field and $p$ does not divide the class number of $k$ (\cite[Proposition 3.A]{minardi}).
Moreover, there is a sufficient condition for GGC to hold in the case where $k$ is a CM-field and $p$ splits completely in $k/\Q$ (\cite[Theorem 2]{fujii}).

In this paper we focus on some consequences of GGC on the size of $X(K)$ for (multiple) $\Z_p$-extensions $K$ of $k$.
To state the main result, recall the definitions of the Iwasawa $\lambda, \mu, \nu$-invariants of a $\Z_p$-extension $K/k$.
Let $k_n$ be the $n$-th layer of $K/k$, in other words, 
the intermediate field of $K/k$ such that $\Gal(K/k_n) = \Gal(K/k)^{p^n}$.
Then there are unique non-negative integers $\lambda(K/k), \mu(K/k)$ and an integer $\nu(K/k)$ 
such that 
\[
\sharp X(k_n) = p^{\lambda(K/k)n + \mu(K/k)p^n + \nu(K/k)}
\]
for sufficiently large $n$ (see \cite[Theorem 13.13]{washington}).
In the case where $k$ is an imaginary quadratic field, the following theorem is known.

\begin{thm}[{\cite[Theorem 2]{ozaki}}]\label{thm_ozaki}
Let $k$ be an imaginary quadratic field.
Put $s = 1$ if $p$ splits in $k$ and $s = 0$ otherwise.
Suppose GGC holds for $(k, p)$.
Then for all but finitely many $\Z_p$-extension $K$ of $k$,
if one of the primes of $k$ above $p$ does not split in $K/k$, then $\mu(K/k) = 0$ and $\lambda(K/k) = s$.
\end{thm}

The main theorem of this paper is Theorem \ref{thm_main}, which gives a partial generalization of Theorem \ref{thm_ozaki} for arbitrary number field $k$.
It is known that the set $\EE(k)$ of all $\Z_p$-extensions of $k$ is equipped with a compact Hausdorff topology (\cite{green2}).
Let $\EE_{\ns}(k)$ be the set of all $\Z_p$-extensions $K$ of $k$ in which every prime of $k$ above $p$ does not split.
Then $\EE_{\ns}(k)$ is an open and closed subset of $\EE(k)$ (Lemma \ref{lem_clopen}).
For any number field $k$ (and the fixed prime $p$), we will define a non-negative integer $s(k)$ in Section \ref{sec_observation}.
As a special case of Theorem \ref{thm_main}, we obtain the following theorem.

\begin{thm}\label{thm_easy}
Let $k$ be an imaginary abelian field.
Suppose GGC holds for $(k, p)$.
Then the set
\[
\{ K \in \EE_{\ns}(k) \mid \mu(K/k) = 0, \lambda(K/k) = s(k) \}
\]
contains an open dense subset of $\EE_{\ns}(k)$.
Moreover, $s(k) = 0$ if $p$ does not split in $k/\Q$ and $s(k) = [k:\Q]/2$ if $p$ splits completely in $k/\Q$.
\end{thm}


The construction of this paper is as follows.
In Section \ref{sec_grassmann}, we define a topology and a measure on the set of $\Z_p^i$-extensions of $k$ for a fixed positive integer $i$.
Although the measure is unnecessary to prove only Theorem \ref{thm_easy},
it enables us to give a stronger statement.
Section \ref{sec_lemmas} is a collection of lemmas about the measure which will be repeatedly used in the later sections.
The proofs of Lemmas \ref{lem_rank} and \ref{lem_induction} are postponed to Section \ref{sec_appendix}.
In Section \ref{sec_observation}, we observe some technical conditions which appear in Section \ref{sec_main_results}.
In Section \ref{sec_main_results}, we state the main theorem of this paper and deduce it from three theorems,
which will be proved in Sections \ref{sec_descent}, \ref{sec_characteristic} and \ref{sec_openness}, respectively.
The contents of Sections \ref{sec_descent}, \ref{sec_characteristic} and \ref{sec_openness} are completely independent of each other.

\section*{Acknowledgment}
The author would like to express his gratitude to his supervisor Prof. Takeshi Tsuji for many helpful advices.
The author is supported by the FMSP program at the University of Tokyo.

\section{$p$-adic Grassmann manifold}\label{sec_grassmann}

In this section we define a topology and a measure on the $p$-adic Grassmann manifold,
which allows us to define a topology and a measure on the set of all $\Z_p^i$-extensions of $k$ for a fixed number field $k$ and a fixed positive integer $i$.

Before the discussion about the Grassmann manifold,
we introduce some general terminologies.

\begin{defn}\label{def_generic}
Let $X$ be a topological space and $\mu$ a Borel measure (i.e., a measure defined for the Borel sets) on $X$.
Let $P(x)$ be a property of $x \in X$.
\begin{enumerate}
\item We say that {\it generic} $x \in X$ satisfy $P$ if there exists a closed subset $E$ of $X$ containing the set $\{x \in X \mid \lnot P(x)\}$ with $\mu(E) = 0$, where $\lnot$ denotes the negation.
\item We say that {\it almost all} $x \in X$ satisfy $P$ if there exists a measurable subset $E$ of $X$ containing the set $\{x \in X \mid \lnot P(x)\}$ with $\mu(E) = 0$.
\item We say that {\it weakly almost all} $x \in X$ satisfy $P$ if $\mu(E)=0$ for any measurable subset $E$ of $X$ contained in the set $\{x \in X \mid \lnot P(x)\}$.
\end{enumerate}
\end{defn}

\begin{rem}\label{rem_generic}
\begin{enumerate}
\item It is obvious that 
\[
\text{generic $x \in X$ satisfy $P$} \Rightarrow
\text{almost all $x \in X$ satisfy $P$} \Rightarrow
\text{weakly almost all $x \in X$ satisfy $P$}.
\]
Moreover, suppose that the measure of any non-empty open subset of $X$ is non-zero (as any measure spaces appeared in this paper).
Then
\[
\text{generic $x \in X$ satisfy $P$} \Rightarrow \text{the set $\{x \in X \mid P(x)\}$ contains an open dense subset of $X$}.
\]
It is a standard fact that the converses do not hold in general.
\item In fact, the term ``almost all'' is introduced in order to justify the term ``weakly almost all''
and will not be used essentially in this paper.
For the reason why we introduced the notion ``weakly almost all,'' see Remark \ref{rem_weakly}.
\end{enumerate}
\end{rem}

The following lemma can be easily proved.
We will often make use of it implicitly.

\begin{lem}\label{lem_both}
Let $X$ be a topological space and $\mu$ a Borel measure on $X$.
Let $P_1(x), P_2(x)$ be two properties of $x \in X$.
\begin{enumerate}
\item[$(1)$] If generic $x \in X$ satisfy $P_1$ and generic $x \in X$ satisfy $P_2$,
then generic $x \in X$ satisfy both $P_1$ and $P_2$.

\item[$(2)$] If almost all $x \in X$ satisfy $P_1$ and 
almost all (resp. weakly almost all) $x \in X$ satisfy $P_2$,
then almost all (resp. weakly almost all) $x \in X$ satisfy both $P_1$ and $P_2$.
\end{enumerate}
\end{lem}

Now we begin the discussion about the $p$-adic Grassmann manifold.
Let $M$ be a free $\Z_p$-module of rank $d$ and $i$ a positive integer with $i \leq d$.

\begin{defn}
We define the $p$-adic Grassmann manifold $\Gr(i, M)$ as 
the set of all $\Z_p$-submodules $N$ of $M$ such that 
$M/N$ is a free $\Z_p$-module of rank $i$.
\end{defn}

We denote by $\Aut(M)$ the group of automorphisms of $M$ as a $\Z_p$-module.
It is well-known that $\Aut(M)$ admits a natural topology defined by choosing a $\Z_p$-basis of $M$ and identifying $\Aut(M)$ with $\GL_d(\Z_p)$. 
This topology is independent of the choice of the basis and makes $\Aut(M)$ a profinite group.

If $g \in \Aut(M)$ and $N \in \Gr(i,M)$, then $M/g(N) = g(M/N)$ shows that $g(N) \in \Gr(i,M)$.
Thus the group $\Aut(M)$ acts on the Grassmann manifold $\Gr(i,M)$ naturally.

\begin{lem}\label{lem_transitivity}
The natural action of $\Aut(M)$ on $\Gr(i,M)$ is transitive.
\end{lem}
\begin{proof}
Let $N$ and $N'$ be any two elements of $\Gr(i,M)$.
Since $M/N$ is a free module, there exists a submodule $L$ of $M$ such that $M = N \oplus L$.
Similarly let $M = N' \oplus L'$.
As the ranks of $N$ and $N'$ are equal, we can construct an automorphism $g$ of $M$
such that $g(N) = N'$ and $g(L) = L'$.
This completes the proof.
\end{proof}

Take the Haar measure on $\Aut(M)$ which is normalized so that the measure of $\Aut(M)$ is 1.
(Since $\Aut(M)$ is compact, the left Haar measure is automatically the right Haar measure,
so we need not mention it.
Note that, because in the following we mind only whether the measure of a certain subset is zero or not, the normalization does not matter at all.)
Take any $N_0 \in \Gr(i, M)$ and consider the surjective map (by Lemma \ref{lem_transitivity})
\[
\begin{array}{ccc}
\Aut(M) & \twoheadrightarrow & \Gr(i, M).\\
g & \mapsto & g(N_0)
\end{array}
\]
By this surjective map, we give the quotient topology and the pushforward measure to $\Gr(i,M)$.
This measure on $\Gr(i, M)$ is a Borel measure and $\Aut(M)$-invariant.

\begin{lem}
The topology and the measure on $\Gr(i,M)$ defined above are independent of the choice of $N_0$.
\end{lem}
\begin{proof}
Take another $N_0' \in \Gr(i,M)$.
By Lemma \ref{lem_transitivity}, there is $g' \in \Aut(M)$ such that $N_0' = g'(N_0)$.
Then we have the following commutative diagram
\[
\begin{CD}
\Aut(M) @>{g \mapsto g(N_0')}>> \Gr(i,M)\\
@V {\bullet g'} VV @VV{id}V\\
\Aut(M) @>>{g \mapsto g(N_0)}> \Gr(i,M).
\end{CD}
\]
Since the left vertical arrow is a homeomorphism preserving the measure, this diagram proves the lemma.
\end{proof}

The topology of $\Gr(i, M)$ can be described as follows.
For $N_0 \in \Gr(i, M)$ and a non-negative integer $n$, put
\[
V_n(N_0) = \{ N \in \Gr(i, M) \mid N + p^n M = N_0 + p^n M \}.
\]
Note that $V_n(N_0) = \{ N \in \Gr(i, M) \mid N \subset N_0 + p^n M \}$.
In fact, if $N \subset N_0 + p^n M$, then $N + p^n M \subset N_0 + p^n M$ and the both sides have the same index $p^n i$ in $M$.
Therefore $N \in V_n(N_0)$, as claimed.

\begin{lem}\label{lem_neighborhood}
$V_n(N_0)$ is an open and closed subset of $\Gr(i, M)$ and 
the family $\{ V_n(N_0)\}_n$ constitute a fundamental system of neighborhoods of $N_0$.
\end{lem}

\begin{proof}
Since $V_0(N_0) = \Gr(i, M)$ is trivially an open and closed subset, we consider positive integers $n$.
Let $\End(M)$ denote the ring of endomorphisms of $M$ as a $\Z_p$-module.
Then $1+p^n\End(M)$ is an open subgroup of $\Aut(M)$.

Let $\alpha : \Aut(M) \to \Gr(i, M)$ be the surjective map defined by $g \mapsto g(N_0)$.
We claim that $\alpha(1+p^n\End(M)) = V_n(N_0)$.
For any $h \in \End(M)$, we have
\[
(1+p^n h)(N_0) \subset N_0 + p^n h(N_0) \subset N_0 + p^n M,
\]
which shows that $\alpha(1+p^n\End(M)) \subset V_n(N_0)$.
Conversely take any $N \in V_n(N_0)$.
Since $N \subset N_0 + p^n M$ and $N$ is a free $\Z_p$-module, there is a $\Z_p$-homomorphism $h: N \to M$ such that
$(1 - p^n h)(N) \subset N_0$.
Since $M/N$ is a free $\Z_p$-module, we can extend $h$ so that $h \in \End(M)$.
Then we have $(1-p^n h)(N) = N_0$ and consequently $\alpha((1-p^n h)^{-1}) = N$,
which proves the claim.

By the definition of the topology on $\Gr(i, M)$, the above claim proves the lemma immediately. 
\end{proof}

Let $k$ be a number field.
Throughout this paper, we usually denote a $\Z_p^i$-extension of $k$ by $K^{(i)}$.
If $i = 1$ then we often omit the superscript and denote a $\Z_p$-extension of $k$ by $K$.
Let $K^{(d)}$ be a $\Z_p^d$-extension of $k$ and $i$ a positive integer with $i \leq d$.
Then we have a natural bijection of sets
\[
\begin{array}{ccc}
\Gr(i, \Gal(K^{(d)}/k)) &\simeq & \{ \Z_p^{i} \text{-extension of }k \text{ contained in } K^{(d)} \}.\\
\Gal(K^{(d)}/K^{(i)}) &\leftrightarrow& K^{(i)}
\end{array}
\]
Through this bijection, we give a topology and a Borel measure on the set of $\Z_p^i$-extensions of $k$ contained in $K^{(d)}$.

\begin{rem}\label{rem_top_greenberg}
Recall that $\EE(k)$ denote the set of all $\Z_p$-extensions of $k$, which is identified with $\Gr(1, \Gal(\tilk/k))$.
Then by Lemma \ref{lem_neighborhood}, a fundamental system of neighborhoods of $K_0 \in \EE(k)$ is
given by $\{ K \in \EE(k) \mid [K \cap K_0:k] \geq p^n \}$ where $n$ runs through non-negative integers.
Therefore the topology on $\EE(k)$ coincides with that defined in \cite{green2}.
\end{rem}

Let $P$ be a property of $\Z_p^i$-extensions of $k$.
We say that {\it generic} (resp. {\it almost all}, resp. {\it weakly almost all}) $\Z_p^i$-extensions $K^{(i)} \subset K^{(d)}$ of $k$ satisfy $P$ if generic (resp. almost all, resp. weakly almost all) $K^{(i)} \in \Gr(i, \Gal(K^{(d)}/k))$ satisfy $P$.
When $K^{(d)} = \tilk$, we simply say that generic (resp. almost all, resp. weakly almost all) $\Z_p^i$-extensions $K^{(i)}$ of $k$ satisfy $P$.

\section{Lemmas on measure}\label{sec_lemmas}
In this section we gather some lemmas, mainly regarding the measure on the $p$-adic Grassmann manifold.

As in the previous section, let $M$ be a free $\Z_p$-module of rank $d$ and $i$ a positive integer with $i \leq d$.
We establish a method to compute the measure.
Choose a $\Z_p$-basis of $M$ and identify $M$ with $\Z_p^d$ whose elements are written as column vectors.
Then we can also identify $\Aut(M)$ with $\GL_d(\Z_p)$ which acts on $\Z_p^d$ by left multiplication.
Let $e_1, \dots, e_d$ be the standard basis of $\Z_p^d$ and 
put $N_0 = \langle e_1, \dots, e_{d-i} \rangle \in \Gr(i, \Z_p^d)$.
Then the isotropy group of $N_0$ with respect to the action of $\GL_d(\Z_p)$ is 
\[
B = \{(g_{jk}) \in \GL_d(\Z_p) \mid g_{jk}=0 \text{ if } j > d-i \text{ and } k \leq d-i \}
= \left\{
\begin{pmatrix}
\ast_{d-i} & \ast \\
0 & \ast_i
\end{pmatrix}
\right\},
\]
where the subscript denotes the size of square matrices.
The map $\GL_d(\Z_p) / B \simeq \Gr(i, \Z_p^d)$ defined by $g \mapsto g(N_0)$ is a homeomorphism preserving the measure.

Define another subgroup $H$ of $\GL_d(\Z_p)$ by
\[
H = \{(g_{jk}) \in \GL_d(\Z_p) \mid g_{jk} = \delta_{jk} \text{ if } j \leq d-i \text{ or } k > d-i \}
=
\left\{
\begin{pmatrix}
1_{d-i} & 0 \\
\ast & 1_i
\end{pmatrix}
\right\},
\]
where $\delta_{jk}$ denotes the Kronecker delta.
Then $H$ is isomorphic to $M_{i,d-i}(\Z_p)$ as a topological group via
\[
\begin{array}{ccc}
M_{i,d-i}(\Z_p) & \simeq & H, \\
A & \leftrightarrow & \begin{pmatrix} 1_{d-i} & 0 \\ A & 1_i \end{pmatrix}
\end{array}
\]
and it gives a parameterization of a neighborhood of $N_0$ as follows.

\begin{lem}\label{lem_measure_calculation}
The natural map $H \to \GL_d(\Z_p) / B$ is a homeomorphism onto an open subset and 
the restriction of the measure of $\GL_d(\Z_p) / B$ to $H$ is a Haar measure on $H$.
\end{lem}

\begin{proof}
The injectivity follows from $H \cap B = \{ 1 \}$.
Hence the map is a homeomorphism onto its image.
We shall show that $HB \subset \GL_d(\Z_p)$ is an open subset.
First observe that $B$ contains
\[
B' = 1_d + p\{(g_{jk}) \in M_d(\Z_p) \mid g_{jk}=0 \text{ if } j > d-i \text{ and } k \leq d-i \}
=\left\{1_d + p
\begin{pmatrix}
\ast_{d-i} & \ast \\
0 & \ast_i
\end{pmatrix}
\right\}
\]
and $H$ contains
\[
H' := 1_d + p\{(g_{jk}) \in M_d(\Z_p) \mid g_{jk}=0 \text{ if } j \leq d-i \text{ or } k > d-i \}
=
\left\{
1_d + p\begin{pmatrix}
0_{d-i} & 0 \\
\ast & 0_i
\end{pmatrix}
\right\}.
\]
One can easily check that
\[
H'B' = 1+pM_d(\Z_p).
\]
Hence for any $h \in H$ and $b \in B$, we have
\[
HB \supset hH'B'b = h(1+ pM_d(\Z_p))b = hb + pM_d(\Z_p),
\]
which is an open neighborhood of $hb$.
This shows that $HB$ is open in $\GL_d(\Z_p)$.
Therefore the image of $H \to \GL_d(\Z_p) / B$ is an open subset.

The restriction of the measure of $\GL_d(\Z_p) / B$ to $H$ is clearly $H$-invariant 
and the openness shows that it is not the zero measure.
For the outer and inner regularity, 
we use the fact that a finite Borel measure on a metrizable space is outer and inner regular.
This proves that the concerned measure is a Haar measure on $H$.
\end{proof}

Therefore the image of $H \to \GL_d(\Z_p)/B \simeq \Gr(i, \Z_p^d)$ is an open neighborhood of $N_0$.
We shall show that $\Gr(i, M)$ is covered by such open sets.
For a set $W \subset \{1, \dots, d\}$ with $d-i$ elements, 
we put $N_W = \langle e_w \mid w \in W \rangle \in \Gr(i, \Z_p)$ (hence $N_0 = N_{\{1, \dots, d-i\}}$).
Then we can construct an open neighborhood $U_W$ of $N_W$ in the same manner as above.
In fact, put 
\[
H_W = \{ (g_{jk}) \in \GL_d(\Z_p) \mid g_{jk} = \delta_{jk} \text{ if } j \in W \text{ or } k \not\in W \}
\]
and let $U_W$ denote the image of the map $H_W \to \Gr(i, \Z_p^d)$ defined by $g \mapsto g(N_W)$.
Then by Lemma \ref{lem_measure_calculation}, $U_W$ is an open neighborhood of $N_W$.
The following lemma can be easily proved, and we omit the proof.

\begin{lem}\label{lem_covering}
The family $\{ U_W \}_W$, where $W$ runs through all subsets of $\{1, \dots, d\}$ with $d-i$ elements, constitute an open covering of $\Gr(i, \Z_p^d)$.
\end{lem}

Lemmas \ref{lem_rank} and \ref{lem_induction} below will play important roles in Sections \ref{sec_observation} and \ref{sec_main_results}, respectively.
Because the proofs of them are elementary but considerably long, we postpone them to Section \ref{sec_appendix}.

\begin{lem}\label{lem_rank}
Let $M$ be a free $\Z_p$-module of rank $d$ and $i$ a positive integer with $i \leq d$.
Let $L_1, \dots, L_r$ be $\Z_p$-submodules of $M$ such that $\rank_{\Z_p} L_j \geq i$ for $1 \leq j \leq r$.
\begin{enumerate}
\item[$(1)$] $\rank_{\Z_p} (\Image(L_j \to M/N)) = i$ for all $1 \leq j \leq r$ for generic $N \in \Gr(i, M)$.
More generally, if $i \leq i' \leq d$ is a positive integer, then $\rank_{\Z_p} (\Image(L_j \to M/N)) \geq i$ for all $1 \leq j \leq r$ for generic $N \in \Gr(i', M)$.
\item[$(2)$] Suppose $i=1$. 
For $N \in \Gr(1, M)$, put 
\[
s(N) = \rank_{\Z_p} \left( N\bigg/ \sum_{j=1}^r (N \cap L_j) \right)
\]
and put 
\[
s = \min \{ s(N) \mid N \in \Gr(1,M), \rank_{\Z_p} (\Image(L_j \to M/N)) = 1 \text{ for all } 1 \leq j \leq r\}.
\] 
Then $s(N) = s$ for generic $N \in \Gr(1, M)$.
\end{enumerate}
\end{lem}

In the following lemma, $K^{(i)}$ always denotes a $\Z_p^i$-extension of $k$.

\begin{lem}\label{lem_induction}
Let $k$ be a number field and let $d, d'$, and $d''$ be positive integers with $d'' \leq d' \leq d$.
Let $K^{(d)}$ be a $\Z_p^d$-extension of $k$.
Let $P$ (resp. $Q$) be a property of $\Z_p^{d'}$-extensions (resp. $\Z_p^{d''}$-extensions) of $k$.
Suppose
\begin{enumerate}
\item[$(a)$] $P(K^{(d')})$ for weakly almost all $K^{(d')} \subset K^{(d)}$, and
\item[$(b)$] for any $K^{(d')} \subset K^{(d)}$, $P(K^{(d')})$ implies $Q(K^{(d'')})$ for weakly almost all $K^{(d'')} \subset K^{(d')}$.
\end{enumerate}
Then $Q(K^{(d'')})$ for weakly almost all $K^{(d'')} \subset K^{(d)}$.
\end{lem}


\section{Numbers $s(k)$ and $s'(k)$}\label{sec_observation}
Let $k$ be a number field.
In this section, we define non-negative integers $s(k)$ and $s'(k)$ concerning the ramifications and the decompositions of primes, respectively.
In fact $s'(k) = 0$ conjecturally.
They will appear in Theorem \ref{thm_main}.

Before the main argument, we recall here some facts from class field theory.
We denote by $S_p(k)$ the set of all primes of $k$ above $p$.
For a subset $S$ of $S_p(k)$, let $M_S(k)$ be the maximal $S$-ramified abelian pro-$p$ extension of $k$.
Let $E_k$ be the unit group of $k$.
For a prime $\pe \in S$, let $k_{\pe}$ be the completion of $k$ at $\pe$,
$U_{\pe}$ the unit group of $k_{\pe}$, 
and $U_{\pe}^{(1)} \subset U_{\pe}$ the principal unit group.
Then we have a natural diagonal map $E_k \to \prod_{\pe \in S} U_{\pe}$.
This map is extended to $E_k \otimes \Z_p \to \prod_{\pe \in S} U_{\pe}^{(1)}$ and 
we will denote the cokernel by $\prod_{\pe \in S} U_{\pe}^{(1)} \bigg/ (E_k \otimes \Z_p)$.

\begin{thm}[see {\cite[Corollary 13.6]{washington}}]\label{thm_cft}
For a number field $k$ and a set $S \subset S_p(k)$, we have
\[
\Gal(M_S(k)/L(k)) \simeq  \prod_{\pe \in S} U_{\pe}^{(1)}  \bigg/ (E_k \otimes \Z_p)
\]
via the Artin map.
For $\pe \in S$, the inertia group of $\pe$ in $\Gal(M_S(k)/L(k))$ corresponds to the image of $U_{\pe}^{(1)}$ in the right hand side.
\end{thm}

Note that $\Gal(M_{S_p(k)}(k)/\tilk)$ is the torsion part of $\Gal(M_{S_p(k)}(k)/k)$ as a finitely generated $\Z_p$-module.

We also introduce the following notations.
When $F'$ is an abelian extension of an algebraic extension $F$ of $\Q$ and $\pe$ is a finite prime of $F$,
we denote by $I_{\pe}(F'/F)$ and $D_{\pe}(F'/F)$ the inertia group and the decomposition group of $\pe$ in $\Gal(F'/F)$, respectively.
Moreover, if $F_1$ is an intermediate field of $F'/F$,
we denote by $I_{\pe}(F'/F_1)$ and $D_{\pe}(F'/F_1)$ the inertia group and the decomposition group of a prime of $F_1$ above $\pe$ in $\Gal(F'/F_1)$, respectively.
The definition is independent of the choice of the prime of $F_1$ above $\pe$ and
in fact $I_{\pe}(F'/F_1) = \Gal(F'/F_1) \cap I_{\pe}(F'/F)$ and $D_{\pe}(F'/F_1) = \Gal(F'/F_1) \cap D_{\pe}(F'/F)$.

We begin the main argument.
Let $S_p(k) = \{ \pe_1, \dots, \pe_r \}$.
Our main task in the rest of this section is to apply Lemma \ref{lem_rank} to the three objects:
\begin{enumerate}
\item[(A)] $M = \Gal(\tilk / k), i = 1$, and $L_j = I_{\pe_j}(\tilk/k)$ $(1 \leq j \leq r)$.
\item[(B)] $M = \Gal(\tilk / k), i = 1$, and $L_j = D_{\pe_j}(\tilk/k)$ $(1 \leq j \leq r)$.
\item[(C)] $M = \Gal(\tilk / k), i = 2$, and $L_j = D_{\pe_j}(\tilk/k)$ $(1 \leq j \leq r)$.

More generally, $M = \Gal(K^{(d)} / k), i = 2$, and $L_j = D_{\pe_j}(K^{(d)}/k)$ $(1 \leq j \leq r)$, where $K^{(d)}/k$ is a $\Z_p^d$-extension of $k$.
\end{enumerate}

Note that for every $j$, the prime $\pe_j$ is ramified in the cyclotomic $\Z_p$-extension $k^{\cyc}$ of $k$ and thus
we have 
\[
\rank_{\Z_p} D_{\pe_j}(\tilk/k) \geq \rank_{\Z_p} I_{\pe_j}(\tilk/k) \geq 1.
\]

\subsection*{Applying to (A)}
Recall that we denote by $\EE(k)$ the set of all $\Z_p$-extensions of $k$.
Put
\[
\EE_{\ram}(k) = \{K \in \EE(k) \mid \text{every $\pe \in S_p(k)$ is ramified in $K/k$} \}.
\]

\begin{defn}\label{def_s_ram}
For $K \in \EE(k)$, put $s(K/k) = \rank_{\Z_p} X(K)_{\Gal(K/k)}$.
Furthermore let $s(k)$ be the minimum of $s(K/k)$ where $K \in \EE_{\ram}(k)$.
\end{defn}

The number $s(k)$ gives a trivial lower bound of the size of $X(K)$ in a sense and
Theorem \ref{thm_main} claims that $X(K)$ of generic $K$ reaches this lower bound.

\begin{prop}\label{prop_almost_ram}
\begin{enumerate}
\item[$(1)$] $K \in \EE_{\ram}(k)$ for generic $K \in \EE(k)$.
\item[$(2)$] We have the equality $s(K/k) = \rank_{\Z_p} \Gal(\tilk \cap L(K)/K)$.
In particular, $s(k) \leq d(k)-1$.
\item[$(3)$] $s(K/k) = s(k)$ for generic $K \in \EE(k)$.
\end{enumerate}
\end{prop}

\begin{proof}
(1) For any $K$ and $1 \leq j \leq r$, $\pe_j$ is ramified in $K/k$ if and only if $\rank_{\Z_p} I_{\pe}(K/k)=1$.
Since $I_{\pe}(K/k)$ is the image of $I_{\pe}(\tilk/k)$ under the restriction map $\Gal(\tilk/k) \to \Gal(K/k)$,
the assertion follows from Lemma \ref{lem_rank} (1) applied to (A).

(2) Since $\Gal(K/k)$ is pro-cyclic, we have $X(K)_{\Gal(K/k)} = \Gal(\LL/K)$,
where $\LL$ is the maximal abelian extension of $k$ contained in $L(K)$.
It is clear that $\LL \subset M_{S_p(k)}(k)$.
Since $M_{S_p(k)}(k)/\tilk$ is a finite extension, 
$\LL \supset \tilk \cap \LL = \tilk \cap L(K)$ is also a finite extension and we obtain the assertion.

(3) By the definition of $L(K)$, we have
\begin{eqnarray*}
\Gal(\tilk \cap L(K)/K) &=& \Gal(\tilk/K) \bigg/ \left(\sum_{j = 1}^r I_{\pe_j}(\tilk/K) \right)\\
 &=& \Gal(\tilk/K) \bigg/ \left(\sum_{j = 1}^r \left( \Gal(\tilk/K) \cap I_{\pe_j}(\tilk/k) \right) \right).
\end{eqnarray*}
Then the assertion follows from (2) and Lemma \ref{lem_rank} (2) applied to (A).
\end{proof}

\begin{eg}\label{eg_s_ram}
\begin{enumerate}
\item If $p$ splits completely in $k/\Q$, then since every $\pe \in S_p(k)$ has degree one, 
Theorem \ref{thm_cft} implies that $\rank_{\Z_p} I_{\pe}(\tilk/k)=1$.
Hence for every $K \in \EE_{\ram}(k)$, $\tilk/K$ is unramified and $s(K/k) = d(k)-1$.
Consequently $s(k) = d(k)-1$.

\item On contrast, if $p$ does not split in $k/\Q$, 
then $\rank_{\Z_p} I_{\pe}(\tilk/k) = d(k)$ for the only one prime $\pe \in S_p(k)$.
Hence every $\Z_p$-extension $K$ satisfies $s(K/k) = 0$
and consequently $s(k) = 0$.

\item As a consequence of above two examples, if $k$ is an imaginary quadratic field, 
$s(k)$ coincides with the $s$ in Theorem \ref{thm_ozaki}.

\item Let $k$ be a complex cubic field.
Since $\rank_{\Z} E_k = 1$, Leopoldt's Conjecture trivially holds and $d(k)=2$.
We shall show that $s(k)=1$ if $p$ splits completely in $k$ and $s(k)=0$ otherwise.
The remained case is $\sharp S_p(k)=2$, so
let $S_p(k) = \{ \pe_1, \pe_2 \}$ with $\deg \pe_1 = 1$.
Then Theorem \ref{thm_cft} implies that $\rank_{\Z_p} I_{\pe_2}(\tilk/k)=2$.
Hence we obtain $s(K/k)=0$ for any $K \in \EE(k)$.

\item If $k$ is a totally imaginary quartic field,
then $d(k)=3$ and, $s(k)=2$ if $\sharp S_p(k)=4$, $s(k)=1$ if $\sharp S_p(k)=3$, and $s(k)=0$ otherwise.
The proof is done in the similar way as the previous example, so we omit it.

Note that, in case $S_p(k) = \{ \pe_1, \pe_2\}$ with $\deg \pe_1 = \deg \pe_2 = 2$,
$s(K/k)$ is not constant for $K \in \EE_{\ram}(k)$.
Indeed, let $K_i$ be the unique $\{\pe_i\}$-ramified $\Z_p$-extension of $k$ for $i=1,2$,
whose unique existence is assured by Theorem \ref{thm_cft}.
If $K \subset K_1K_2$, then $K_1K_2/K$ is unramified and $\tilk/K_1K_2$ is not, hence $s(K/k) = 1$.
On contrast, if $K \not\subset K_1K_2$, then $s(K/k) = 0$.
To see this, assume contrary there exists an unramified $\Z_p$-extension $K^{(2)}$ of $K$ contained in $\tilk$.
Then $K^{(2)}$ must contain both $K_1$ and $K_2$, which leads to $K_1K_2 \supset K$, a contradiction.
\end{enumerate}
\end{eg}

\subsection*{Applying to (B)}
The general theory proceeds completely in parallel with (A).
For an algebraic extension $F$ of $\Q$, let $L'(F)$ be the maximal unramified pro-$p$ abelian extension of $F$ in which every prime of $F$ above $p$ splits completely
and let $X'(F)$ be the Galois group $\Gal(L'(F)/F)$.
Obviously we have $L'(F) \subset L(F)$.

Put
\[
\EE_{\sfin}(k) = \{K \in \EE(k) \mid \text{every $\pe \in S_p(k)$ splits finitely in $K/k$} \} \supset \EE_{\ram}(k).
\]
Similarly as in Definition \ref{def_s_ram}, for $K \in \EE(k)$ we put $s'(K/k) = \rank_{\Z_p} X'(K)_{\Gal(K/k)}$
and let $s'(k)$ be the minimum of $s'(K/k)$ where $K \in \EE_{\sfin}(k)$.
The following proposition can be obtained exactly in the same manner as Proposition \ref{prop_almost_ram},
applying Lemma \ref{lem_rank} to (B).

\begin{prop}\label{prop_almost_dec}
\begin{enumerate}
\item[$(1)$] $K \in \EE_{\sfin}(k)$ for generic $K \in \EE(k)$.
\item[$(2)$] We have the equality $s'(K/k) = \rank_{\Z_p} \Gal(\tilk \cap L'(K)/K)$.
In particular, $s'(k) \leq d(k)-1$.
\item[$(3)$] $s'(K/k) = s'(k)$ for generic $K \in \EE(k)$.
\end{enumerate}
\end{prop}

On contrast to $s(k)$, conjecturally $s'(k)$ vanishes.
More precisely, we have the following conjecture (see \cite[Remarques (i) after Proposition 6]{jaulent}).

\begin{conj}[{Generalized Gross' Conjecture}]\label{conj_gross}
$X'(k^{\cyc})_{\Gal(k^{\cyc}/k)}$ is finite, in other words,  $s'(k^{\cyc}/k) = 0$.
\end{conj}
In particular, Conjecture \ref{conj_gross} implies that $s'(k) = 0$.
It is known that Conjecture \ref{conj_gross} holds if $k/\Q$ is abelian (\cite{green4}).

We say that a $\Z_p$-extension $K$ of $k$ is {\it arithmetically semi-simple} 
if $K \in \EE_{\sfin}(k)$ and $s'(K/k) = 0$ (\cite[Definition 7]{jaulent}).
Thus Conjecture \ref{conj_gross} asserts that $k^{\cyc}/k$ is arithmetically semi-simple.
Note that in general there exist $\Z_p$-extensions of $k$ in $\EE_{\sfin}(k)$ which are not arithmetically semi-simple
even if $k/\Q$ is abelian (see \cite{kisilevsky}, for example).
This terminology comes from the following lemma.

\begin{lem}[{\cite[Proposition 6]{jaulent}}]\label{lem_jaulent}
Let $K$ be an arithmetically semi-simple $\Z_p$-extension of $k$
and let $\sigma$ be a generator of $\Gal(K/k)$.
Then the module $X(K)$ is semi-simple at $\sigma-1$.
In other words, if $\bigoplus_i (\Lambda(K/k)/(f_i))$ is an elementary module pseudo-isomorphic to $X(K)$, 
then $(\sigma-1)^2$ does not divide any of $f_i$.
\end{lem}

Proposition \ref{prop_almost_dec} immediately implies the following.

\begin{cor}\label{cor_jaulent}
If $s'(k) = 0$, then generic $\Z_p$-extensions of $k$ are arithmetically semi-simple.
\end{cor}

\subsection*{Applying to (C)}
In order to apply Lemma \ref{lem_rank} (1) to (C), we need the following condition.

\begin{ass}\label{ass_decomp}
For every $\pe \in S_p(k)$, we have $\rank_{\Z_p} D_{\pe}(\tilk/k) \geq 2$.
\end{ass}

Clearly $d(k) \geq 2$ if Assumption \ref{ass_decomp} holds.
Conversely, the author does not know any counter-example of Assumption \ref{ass_decomp} if $d(k) \geq 2$
(see \cite[Remarque 3.3]{nqd}).
We give a sufficient condition for Assumption \ref{ass_decomp}.

\begin{lem}\label{rem_decomp}
If $k/\Q$ is imaginary Galois, then Conjecture \ref{conj_gross} implies Assumption \ref{ass_decomp}.
In particular, if $k/\Q$ is imaginary abelian, then Assumption \ref{ass_decomp} holds.
\end{lem}

\begin{proof}
Since $k/\Q$ is a Galois extension, the decomposition groups $D_{\pe}(\tilk/k)$ for $\pe \in S_p(k)$ are isomorphic to each other.
Therefore, if Assumption \ref{ass_decomp} fails, then we have $\rank_{\Z_p} D_{\pe}(\tilk/k) = 1$ for every $\pe \in S_p(k)$.
Since $\rank_{\Z_p} D_{\pe}(k^{\cyc}/k) = 1$, we have $D_{\pe}(\tilk/k^{\cyc}) = 0$ for every $\pe \in S_p(k)$, which means that $\tilk \subset L'(k^{\cyc})$.
Then by Proposition \ref{prop_almost_dec} (2), $s'(k^{\cyc}/k) = d(k) - 1 \geq r_2(k) \geq 1$, which contradicts Conjecture \ref{conj_gross}.
\end{proof}

Applying Lemma \ref{lem_rank} (1) to (C), we obtain the following.

\begin{prop}\label{prop_almost_decomp_2}
If Assumption \ref{ass_decomp} holds, 
then $p$ splits finitely in generic $\Z_p^2$-extensions of $k$.
More generally, let $K^{(d)}$ be a $\Z_p^d$-extension of $k$ and $i$ an integer with $2 \leq i \leq d$.
If $\rank_{\Z_p} D_{\pe}(K^{(d)}/k) \geq 2$ for every $\pe \in S_p(k)$, 
then generic $\Z_p^i$-extensions $K^{(i)} \subset K^{(d)}$ of $k$ satisfy the same property. 
\end{prop}

\section{Main results}\label{sec_main_results}

In this section we state the main theorem of this paper (Theorem \ref{thm_main}) and deduce it from three theorems (Theorems \ref{thm_descent_step}, \ref{thm_characteristic}, and \ref{thm_openness}) whose proofs will be given in the later sections.

Let $k$ be a number field.
As already defined in Section \ref{sec_introduction}, put
\[
\EE_{\ns}(k) = \{ K \in \EE(k) \mid \text{every $\pe \in S_p(k)$ does not split in $K/k$} \} \subset \EE_{\sfin}(k).
\]

\begin{lem}\label{lem_clopen}
$\EE_{\ns}(k)$ is an open and closed subset of $\EE(k)$.
\end{lem}
\begin{proof}
For any $K \in \EE(k)$, $K \in \EE_{\ns}(k)$ if and only if every $\pe \in S_p(k)$ does not split in the first layer of $K/k$.
In particular whether $K \in \EE_{\ns}(k)$ or not is determined by the first layer of $K$.
Now Remark \ref{rem_top_greenberg} implies the assertion.
\end{proof}

\begin{rem}
Since $\Q^{\cyc}/\Q$ is a $\Z_p$-extension which is totally ramified at $p$, 
if $p \nmid [k : \Q]$ or $p$ is unramified in $k/\Q$, then $k^{\cyc} \in \EE_{\ns}(k)$ and in particular $\EE_{\ns}(k) \neq \emptyset$.
\end{rem}

Now we can state the main theorem of this paper.

\begin{thm}\label{thm_main}
Suppose GGC holds for $(k, p)$ and $s'(k) = 0$.
Then for generic $\Z_p$-extensions $K$ of $k$, if $K \in \EE_{\ns}(k)$ then $\mu(K/k) = 0$ and $\lambda(K/k) = s(k)$.
\end{thm}

\begin{rem}
Let us illustrate the reason why the $\Z_p$-extensions are restricted to $K \in \EE_{\ns}(k)$.
Consider the extreme case, namely, suppose that $K \in \EE_{\ram}(k)$ satisfies that $p$ splits completely in $k_1/\Q$, where $k_1$ is the first layer of $K/k$.
For simplicity, suppose that Leopoldt's Conjecture holds for $(k, p)$ and $d(k) = 1+r_2(k) \geq 2$.
Then by Theorem \ref{thm_cft}, $\rank_{\Z_p} I_{\pe_1}(\widetilde{k_1}/k_1) = 1$ for every $\pe_1 \in S_p(k_1)$
and hence $\widetilde{k_1}/K$ is unramified.
Consequently, 
\[
\rank_{\Z_p} X(K) \geq \rank_{\Z_p} \Gal(\widetilde{k_1}/K) = d(k_1)-1 \geq r_2(k_1) = r_2(k) p = s(k) p > s(k),
\]
where the last equality comes from Example \ref{eg_s_ram}.
Therefore the conclusion of Theorem \ref{thm_main} does not hold in this case.

We also remark that the assumption $K \in \EE_{\ns}(k)$ is too restrictive.
In fact, by modifying Lemma \ref{lem_cft}, 
one may increase the $\Z_p$-extensions to which Theorem \ref{thm_main} applies
(see Theorem \ref{thm_ozaki}), 
but we do not try the refinement in this paper.
\end{rem}

To state the next theorems, we define an auxiliary algebra as follows.
Let $K/k$ be a $\Z_p$-extension.
We put
\[
\Lambda^{\dagger}(K/k) = \Lambda(K/k) \left[\frac{1}{\gamma - 1} \middle| \gamma \in \Gal(K/k), \gamma \neq 1 \right],
\]
which is a noetherian integrally closed domain.
If $\sigma$ is a topological generator of $\Gal(K/k)$, then
\[
\Lambda^{\dagger}(K/k) = \Lambda(K/k) \left[\frac{1}{\sigma^{p^n} - 1} \middle| \text{$n$ is a non-negative integer} \right].
\]
Over the algebra $\Lambda^{\dagger}(K/k)$, since its dimension is 1, a module is pseudo-null if and only if it is zero and 
a homomorphism is pseudo-isomorphic if and only if it is isomorphic.
In the following, we prefer the term pseudo-null (resp. pseudo-isomorphic) rather than zero (resp. isomorphic) in order to keep harmony with multiple $\Z_p$-extensions.

For a $\Lambda(K/k)$-module $X$, 
we put $X^{\dagger} =  \Lambda^{\dagger}(K/k)\otimes_{\Lambda(K/k)} X$, which is always considered as a $\Lambda^{\dagger}(K/k)$-module.

\begin{lem}\label{lem_dagger}
Let $K/k$ be a $\Z_p$-extension and $X$ a finitely generated torsion $\Lambda(K/k)$-module.
Then $X^{\dagger} \sim 0$ if and only if the characteristic ideal $\ch (X)$ contains $(\gamma - 1)^N$ for
some $\gamma \in \Gal(K/k), \gamma \neq 1$ and some positive integer $N$.
\end{lem}

\begin{proof}
In general, if $\Lambda$ is a noetherian integrally closed domain, $S$ is a multiplicative set of $\Lambda$, and $X$ is a pseudo-null $\Lambda$-module, 
then one can easily show that $S^{-1}X$ is a pseudo-null $S^{-1}\Lambda$-module.
Therefore in our case if $X \sim 0$ then $X^{\dagger} \sim 0$.
By the definition of the characteristic ideal and the flatness of $X^{\dagger}$, the assertion is now deduced to the case where 
$X \simeq \Lambda(K/k)/(f)$ with $f$ a power of an irreducible element of $\Lambda(K/k)$.
In that case 
\[
X^{\dagger} \sim 0 \Leftrightarrow f \in \left(\Lambda(K/k)^{\dagger}\right)^{\times}
\Leftrightarrow \text{$f$ divides $(\gamma - 1)^N$ for
some $\gamma$ and $N$},
\]
which proves the lemma.
\end{proof}

In order to simplify the notation, for a $\Z_p^i$-extension $K^{(i)}$ of $k$ with $i \geq 2$,
we put $\Lambda^{\dagger}(K^{(i)}/k) = \Lambda(K^{(i)}/k)$.
Moreover, for a $\Lambda(\KKi / k)$-module $X$,
we put $X^{\dagger} = \Lambda^{\dagger}(K^{(i)}/k) \otimes_{\Lambda(K^{(i)}/k)} X = X$.

We prove the following theorem in Section \ref{sec_descent}.

\begin{thm}\label{thm_descent_step}
Let $i \geq 1$ and $K^{(i+1)}$ be a $\Z_p^{i+1}$-extension of $k$.
Suppose that 
\[
X(K^{(i+1)}) \sim \bigoplus_{l = 1}^t \Lambda(K^{(i+1)}/k) / (f_l),
\]
where $f_l$ is a nonzero element of $\Lambda(K^{(i+1)}/k)$.
(Such a pseudo-isomorphism always exists since $X(K^{(i+1)})$ is a finitely generated torsion $\Lambda(K^{(i+1)}/k)$-module.)
\begin{enumerate}
\item[$(1)$] For generic $\Z_p^i$-extensions $K^{(i)} \subset K^{(i+1)}$ of $k$, we have a $\Lambda^{\dagger}(K^{(i)}/k)$-homomorphism
\[
\bigoplus_{l = 1}^t \Lambda^{\dagger}(K^{(i)}/k) / (\overline{f_l}) \to X(K^{(i)})^{\dagger}
\]
with pseudo-null cokernel, where $\overline{f_l}$ denotes the natural image of $f_l$.
In particular, if $X(K^{(i+1)}) \sim 0$, then for generic $\Z_p^i$-extensions $K^{(i)} \subset K^{(i+1)}$ of $k$ we have $X(K^{(i)})^{\dagger} \sim 0$.
\item[$(2)$] Suppose that for every $\pe \in S_p(k)$, we have $\rank_{\Z_p} D_{\pe}(K^{(i+1)}/k) \geq 2$.
Then for generic $\Z_p^i$-extensions $K^{(i)} \subset K^{(i+1)}$ of $k$, we have 
\[
X(K^{(i)})^{\dagger} \sim \bigoplus_{l = 1}^t \Lambda^{\dagger}(K^{(i)}/k) / (\overline{f_l}).
\]
\end{enumerate}
\end{thm}

Note that $\overline{f_l}$ is nonzero for all but finitely many $\Z_p^i$-extension $\KKi \subset \KKip$ of $k$ (Lemma \ref{lem_module_4}).

\begin{rem}
In fact, we need only the last part of (1) to prove Theorem \ref{thm_main}.
The general assertion and the proof of it are also valid for tamely ramified Iwasawa modules in the sense of \cite{IMO}.
In the tamely ramified case, it seems that the pseudo-nullity of the Iwasawa module of $\tilk$ often fails and Theorem \ref{thm_descent_step} should play an interesting role.
\end{rem}

We prove the following theorem in Section \ref{sec_characteristic}.

\begin{thm}\label{thm_characteristic}
Let $K \in \EE_{\ns}(k)$.
Suppose that $X(K)^{\dagger} \sim 0$ and
$K/k$ is arithmetically semi-simple (i.e., $s'(K/k) = 0$).
Then $X(K)$ is a finitely generated $\Z_p$-module of rank $s(K/k)$.
\end{thm}

We prove the following theorem in Section \ref{sec_openness}.
\begin{thm}\label{thm_openness}
The set 
\[
\{ K \in \EE_{\ram}(k) \mid \text{$X(K)$ is a finitely generated $\Z_p$-module of rank $s(k)$} \}
\]
is an open subset of $\EE(k)$.
\end{thm}

In the rest of this section, we assume Theorems \ref{thm_descent_step}, \ref{thm_characteristic}, and  \ref{thm_openness}.

\begin{proof}[Proof of Theorem \ref{thm_main}]
Recall that for $K \in \EE(k)$, $\mu(K/k)=0$ if and only if $X(K)$ is finitely generated over $\Z_p$ 
and in that case $\lambda(K/k) = \rank_{\Z_p} X(K)$ (see \cite[Proposition 13.23 and Proposition 13.25]{washington}).
Note that
\begin{itemize}
\item $X(K)^{\dagger} \sim 0$ for weakly almost all $K \in \EE(k)$
by $X(\tilk) \sim 0$, Theorem \ref {thm_descent_step} (1), and Lemma \ref{lem_induction},
\item $s(K/k) = s(k)$ for generic $K \in \EE(k)$ by Proposition \ref{prop_almost_ram} (3),
\item $K/k$ is arithmetically semi-simple for generic $K \in \EE(k)$ by $s'(k) = 0$ and Corollary \ref{cor_jaulent}.
\end{itemize}
By Lemma \ref{lem_both} and Remark \ref{rem_generic}, all of the above properties simultaneously hold for weakly almost all $K \in \EE(k)$.
Therefore by Theorem \ref{thm_characteristic}, 
for weakly almost all $K \in \EE(k)$, if $K \in \EE_{\ns}(k)$ then $X(K)$ is a finitely generated $\Z_p$-module of rank $s(k)$.
Finally Theorem \ref{thm_openness} implies the assertion.
\end{proof}

\begin{proof}[Proof of Theorem \ref{thm_easy}]
Since $k$ is abelian, $s'(k) = 0$ as already remarked after Conjecture \ref{conj_gross}.
Therefore Theorem \ref{thm_main} and Remark \ref{rem_generic} implies the first assertion.
As explained in Example \ref{eg_s_ram}, $s(k) = 0$ if $p$ does not split in $k/\Q$ and
$s(k) = d(k)-1$ if $p$ splits completely in $k/\Q$. 
It is known that if $k$ is abelian, Leopoldt's Conjecture holds, namely the Leopoldt's defect $\delta(k, p) = 0$ (\cite{brumer}).
Hence $d(k) = [k:\Q]/2 + 1$, which completes the proof of Theorem \ref{thm_easy}.
\end{proof}

As other applications of Theorem \ref{thm_descent_step}, we obtain the following corollaries.

\begin{cor}\label{cor_rank_2}
Suppose $d(k) \geq 2$.
If GGC holds for $(k, p)$,
then $X(K^{(2)}) \sim 0$ for weakly almost all $\Z_p^2$-extensions $K^{(2)}$ of $k$.
\end{cor}

\begin{proof}
This corollary follows from Theorem \ref{thm_descent_step} (1) using Lemma \ref{lem_induction} inductively.
\end{proof}

\begin{cor}\label{cor_equivalence}
Suppose $d(k) \geq 2$.
The following are equivalent.
\begin{enumerate}
\item[$(i)$] GGC holds for $(k, p)$ and Assumption \ref{ass_decomp} holds.
\item[$(ii)$] $X(K^{(2)}) \sim 0$ and $p$ splits finitely in $K^{(2)}/\Q$ for weakly almost all $\Z_p^2$-extensions $K^{(2)}$ of $k$.
\item[$(iii)$] $X(K^{(2)}) \sim 0$ and $p$ splits finitely in $K^{(2)}/\Q$ for at least one $\Z_p^2$-extension $K^{(2)}$ of $k$.
\end{enumerate}
\end{cor}

\begin{proof}
The implication $(i) \Rightarrow (ii)$ follows from the combination of Corollary \ref{cor_rank_2} and Proposition \ref{prop_almost_decomp_2}, using Lemma \ref{lem_both}.
It is trivial that $(ii) \Rightarrow (iii)$.
\cite[Proposition 4.B]{minardi} shows that $(iii) \Rightarrow (i)$.
\end{proof}

\begin{cor}
Let $K^{(d)}$ be a $\Z_p^d$-extension of $k$ such that $\rank_{\Z_p} D_{\pe}(K^{(d)}/k) \geq 2$ for every $\pe \in S_p(k)$.
Suppose that $X(K^{(d)}) \sim \bigoplus_{l=1}^t \Lambda(K^{(d)}/k)/(f_l)$.
Then $X(K)^{\dagger} \sim \bigoplus_{l=1}^t \Lambda^{\dagger}(K/k)/(\overline{f_l})$ for weakly almost all $\Z_p$-extensions $K \subset K^{(d)}$ of $k$,
where $\overline{f_l}$ denotes the natural image of $f_l$.
\end{cor}

\begin{proof}
This corollary follows from Theorem \ref{thm_descent_step} (2) and Proposition \ref{prop_almost_decomp_2} using Lemma \ref{lem_induction} inductively.
\end{proof}

\section{Proof of Theorem \ref{thm_descent_step}}\label{sec_descent}
The most part of the proof of Theorem \ref{thm_descent_step} consists of module theoretic arguments.
See \cite[\S 6]{matsumura} for the basic materials such as primary decompositions.

\begin{prop}\label{prop_module}
Let $\Lambda$ be a regular local ring and 
$S$ an element of $\Lambda$ such that $\Lambda/S \Lambda$ is again a regular local ring.
Let $X$ be a finitely generated $\Lambda$-module such that $\height(\Ann_{\Lambda}(X)) \geq 2$.
Take a shortest primary decomposition $Y_1 \cap \dots \cap Y_r = 0$ of the $\Lambda$-submodule $0 \subset X$
and put $P_j = \sqrt{\Ann_{\Lambda}(X/Y_j)}$, which are distinct associated primes of $X$.

\begin{enumerate}
\item[$(1)$] Define a ${\Lambda}$-module $Z$ by the exact sequence
\[
0 \to X \to \bigoplus_{j=1}^r X/Y_j \to Z \to 0.
\]
Then we have $\height(\Ann_{\Lambda} (Z)) \geq 3$.
\item[$(2)$] We have pseudo-isomorphisms
\[
X[S] \sim_{\Lambda/S \Lambda} \bigoplus_{j=1}^r (X/Y_j)[S]
\]
and
\[
X/SX \sim_{\Lambda/S \Lambda} \bigoplus_{j=1}^r (X/Y_j)/S(X/Y_j),
\]
where $X[S] = \{ x \in X \mid Sx = 0 \}$ and so on.
\end{enumerate}
\end{prop}

\begin{proof}
(1) This assertion is a direct generalization of \cite[Lemma 2]{ozaki}.
Let $P$ be any prime ideal of ${\Lambda}$ with $\height(P) \leq 2$ and we show that $Z_P = 0$, 
where $Z_P$ denotes the localization of $Z$ at $P$.
Observe that
\[
(X/Y_j)_P \neq 0 \Leftrightarrow P \supset \Ann_{\Lambda}(X/Y_j) \Leftrightarrow P \supset P_j \Leftrightarrow P = P_j,
\]
where in the last equivalence we used that $\height(P) \leq 2 \leq \height(P_j)$.
Hence $P \neq P_j$ implies that $(X/Y_j)_P = 0$.
Therefore the localization of the given short exact sequence at $P$ implies that $Z_P = 0$, as asserted.

(2) The snake lemma applied to the short exact sequence in (1) induces an exact sequence of $\Lambda/S \Lambda$-modules
\[
0 \to X[S] \to \bigoplus_{j = 1}^r (X/Y_j)[S] \to Z[S] \to X/SX \to \bigoplus_{j=1}^r (X/Y_j)/S(X/Y_j) \to Z / SZ \to 0.
\]
Since
\[
\Ann_{\Lambda/S\Lambda}(Z[S]) = (\Ann_{\Lambda}(Z[S]))/(S)
\]
and $\Lambda$ is a catenary ring, $\height (\Ann_{\Lambda}(Z[S])) \geq 3$ implies that 
$\height(\Ann_{\Lambda/S\Lambda}(Z[S])) \geq 2$.
Similarly we have $\height(\Ann_{\Lambda/S\Lambda}(Z/SZ)) \geq 2$.
This completes the proof.
\end{proof}

\begin{prop}\label{prop_module_2}
In the situation in Proposition \ref{prop_module},
suppose furthermore that $P = \sqrt{\Ann_{\Lambda}(X)}$ is a prime ideal of $\Lambda$.
\begin{enumerate}
\item[$(1)$] We have $\sqrt{\Ann_{\Lambda}(X/SX)} = \sqrt{P+S \Lambda}$.
\item[$(2)$] $\height( \Ann_{\Lambda/S\Lambda}(X/S X)) \leq 1$ if and only if $\height (P) = 2$ and $S \in P$.
\item[$(3)$] $\height( \Ann_{\Lambda/S\Lambda}(X[S])) \leq 1$ if and only if $\height (P) = 2$ and $S \in P$.
\end{enumerate}
\end{prop}

\begin{proof}
(1) The inclusion $\supset$ is clear by definition.
For the other inclusion, we take any element $a \in \sqrt{\Ann_{\Lambda}(X/SX)}$.
Then there is a positive integer $N$ such that $a^N X \subset SX$.
A generalization of Cayley-Hamilton's theorem shows that 
there are a positive integer $N'$ and elements $c_1, \dots, c_{N'} \in S \Lambda$
such that $(a^N)^{N'} + c_1 (a^N)^{N'-1} + \dots + c_{N'} \in \Ann_{\Lambda}(X) \subset P$.
Therefore $a^{NN'} \in P + S \Lambda$ and $a \in \sqrt{P+S \Lambda}$, as claimed.

(2) This proposition is a direct generalization of \cite[Lemma 3]{ozaki}.
We have by (1)
\[
\sqrt{\Ann_{\Lambda/S\Lambda}(X/SX)} = \sqrt{\Ann_{\Lambda}(X/SX)}/S \Lambda = \sqrt{P+S \Lambda}/S \Lambda,
\]
which proves the assertion.

(3) As in (1), it is clear that $\sqrt{\Ann_{\Lambda}(X[S])} \supset \sqrt{P+(S)}$,
which implies the ``only if'' part.
But the other inclusion does not hold in general
(as a counter-example, consider $\Lambda = \Z_p[[T_1, T_2]], S = T_1$, and 
$X = \Z_p[[T_1, T_2]]/(p, T_2) = \F_p[[T_1]]$).

In order to prove the ``if'' part, we show that if $S \in P$ then $\sqrt{\Ann_{\Lambda}(X[S])} = P$.
Since $S \in P = \sqrt{\Ann_{\Lambda}(X)}$, there is a positive integer $N$ such that $S^N \in \Ann_{\Lambda}(X)$.
Consider the filtration
\[
0 \subset X[S] \subset X[S^2] \subset \dots \subset X[S^N] = X
\]
of $X$.
It can easily shown that $\Ann_{\Lambda}(X[S]) \subset \Ann_{\Lambda}(X[S^m]/X[S^{m-1}])$ for any positive integer $m$.
In fact, for any $a \in \Ann_{\Lambda}(X[S])$ and $x \in X[S^m]$, $S^{m-1} x \in X[S]$ implies that 
$S^{m-1} a x = a S^{m-1} x = 0$, which means that $ax \in X[S^{m-1}]$.
Therefore we have $\Ann_{\Lambda}(X[S])^N \subset \Ann_{\Lambda}(X)$, which shows that $\sqrt{\Ann_{\Lambda}(X[S])} \subset P$, as claimed.
\end{proof}

\begin{lem}\label{lem_module_3}
Let $i$ be a positive integer and put $\Lambda = \Z_p[[T_1, \dots, T_{i+1}]]$.
For $\alpha = (\alpha_2, \dots, \alpha_{i+1}) \in \Z_p^i$,
put $S_{\alpha} = (1+T_1)(1+T_2)^{\alpha_2} \dots (1+T_{i+1})^{\alpha_{i+1}} - 1 \in \Lambda$.
Let $P$ be a prime ideal of $\Lambda$ with $\height(P) = 2$.
If $i=1$, suppose that $P \not \supset ((1+T_1)^{p^N}-1, (1+T_2)^{p^N})$ for any positive integer $N$.
Then $S_{\alpha} \not\in P$ for generic $\alpha \in \Z_p^i$.
\end{lem}

\begin{proof}
Put $A = \{ \alpha \in \Z_p^i \mid S_{\alpha} \in P \}$.
If $A$ is empty, we have nothing to do.
Suppose $A$ is non-empty and choose an element $\alpha' = (\alpha_2', \dots, \alpha_{i+1}') \in A$.
For any $\alpha \in \Z_p^i$, we have
\begin{eqnarray*}
S_{\alpha} - S_{\alpha'} &=& ((1+T_1) (1+T_2)^{\alpha_2} \dots (1+T_{i+1})^{\alpha_{i+1}} - 1) - ( (1+T_1) (1+T_2)^{\alpha_2'} \dots (1+T_{i+1})^{\alpha_{i+1}'} - 1)\\
&=& (1+T_1) (1+T_2)^{\alpha_2'} \dots (1+T_{i+1})^{\alpha_{i+1}'}\left((1+T_2)^{\alpha_2-\alpha_2'} \dots (1+T_{i+1})^{\alpha_{i+1}-\alpha_{i+1}'}-1\right).
\end{eqnarray*}
Put $B = \{ \beta = (\beta_2, \dots, \beta_{i+1}) \in \Z_p^i \mid (1+T_2)^{\beta_2} \dots (1+T_{i+1})^{\beta_{i+1}}-1 \in P\}$.
Obviously $B$ is a $\Z_p$-submodule of $\Z_p^i$ and
the above calculation shows that $A = B + \alpha'$.

We shall show that the index of $B$ in $\Z_p^i$ is infinite.
If not, there exists a positive integer $N$ such that $p^N\Z_p^i \subset B$.
This implies that $(1+T_j)^{p^N}-1 \in P$ for $2 \leq j \leq i+1$.
Since $S_{\alpha'} \in P$, it also follows that $(1+T_1)^{p^N}-1 \in P$.
Consequently $P \supset ((1+T_1)^{p^N}-1, \dots, (1+T_{i+1})^{p^N}-1)$,
which yields a contradiction.
\end{proof}

Let $k$ be a number field, $i$ a positive integer and $K^{(i+1)}$ a $\Z_p^{i+1}$-extension of $k$.
Let $X$ be a finitely generated $\Lambda(K^{(i+1)}/k)$-module.
For a $\Z_p^i$-extension $\KKi \subset \KKip$ of $k$,
we regard the coinvariant $X_{\Gal(K^{(i+1)}/K^{(i)})}$ and the invariant $X^{\Gal(K^{(i+1)}/K^{(i)})}$ as $\Lambda(\KKi/k)$-modules.

\begin{lem}\label{lem_module}
Suppose that $X$ is a pseudo-null $\Lambda(K^{(i+1)}/k)$-module.
Then $(X_{\Gal(K^{(i+1)}/K^{(i)})})^{\dagger} \sim 0$ and $(X^{\Gal(K^{(i+1)}/K^{(i)})})^{\dagger} \sim 0$ for generic $\Z_p^i$-extensions $K^{(i)} \subset K^{(i+1)}$ of $k$.
\end{lem}

\begin{proof}
Choose a $\Z_p$-basis $\sigma_1, \dots, \sigma_{i+1}$ of $\Gal(K^{(i+1)}/k)$
and identify $\Lambda(K^{(i+1)}/k)$ with $\Z_p[[T_1, \dots, T_{i+1}]]$ so that $\sigma_j$ corresponds to $1 + T_j$.
For each $\alpha = (\alpha_2, \dots, \alpha_{i+1}) \in \Z_p^i$, 
let $K_{\alpha}$ be the sub $\Z_p^i$-extension of $K^{(i+1)}$ 
defined as the fixed field of $\langle \sigma_1 \sigma_2^{\alpha_2} \dots \sigma_{i+1}^{\alpha_{i+1}} \rangle$.
Then the map of Lemma \ref{lem_measure_calculation} is read as
\[
\begin{array}{ccccccc}
\Z_p^i &\hookrightarrow& \GL_{i+1}(\Z_p) &\twoheadrightarrow& \Gr(i,\Gal(K^{(i+1)}/k)) &=& \{ K^{(i)} \subset K^{(i+1)} \}.\\
\alpha & \mapsto & 
\begin{pmatrix}
1 & \begin{matrix}0 & \cdots & 0 \end{matrix} \\
\begin{matrix} \alpha_2 \\ \vdots \\ \alpha_{i+1} \end{matrix} & 1_i
\end{pmatrix}
& \mapsto & \langle \sigma_1  \sigma_2^{\alpha_2} \dots \sigma_{i+1}^{\alpha_{i+1}} \rangle
& \leftrightarrow &K_{\alpha}
\end{array}
\]
By Lemmas \ref{lem_measure_calculation} and \ref{lem_covering}, it is enough to show that 
$(X_{\Gal(K^{(i+1)}/K_{\alpha})})^{\dagger} \sim 0$ and 
$(X^{\Gal(K^{(i+1)}/K_{\alpha})})^{\dagger} \sim 0$
 for almost all $\alpha \in \Z_p^i$ with respect to the natural measure on $\Z_p^i$.

Put $S_{\alpha} = \sigma_1 \sigma_2^{\alpha_2} \dots \sigma_{i+1}^{\alpha_{i+1}} - 1 \in \Lambda(K^{(i+1)}/k)$.
Then $\Lambda(K_{\alpha}/k) = \Lambda(K^{(i+1)}/k)/(S_{\alpha}) = \Z_p[[T_2, \dots, T_{i+1}]]$ naturally and we have $X_{\Gal(K^{(i+1)}/K^{(i)})} = X/S_{\alpha} X$ and $X^{\Gal(K^{(i+1)}/K^{(i)})} = X[S_{\alpha}]$.
By Proposition \ref{prop_module} (2) and Lemma \ref{lem_both}, the assertions of this lemma is reduced to the case where $P = \sqrt{\Ann(X)}$ is a prime ideal.
By Proposition \ref{prop_module_2} (2)(3), we can suppose that $\height(P)=2$.

If $i \geq 2$ or $P \not \supset ((1+T_1)^{p^N}-1, (1+T_2)^{p^N}-1)$ for any positive integer $N$,
then by Lemma \ref{lem_module_3}, $S_{\alpha} \not\in P$ for generic $\alpha \in \Z_p^i$.
Therefore the assertion follows from Proposition \ref{prop_module_2} (2)(3) in this case.

Suppose that $i = 1$ and $P \supset ((1+T_1)^{p^N}-1, (1+T_2)^{p^N}-1)$ for some positive integer $N$.
Then by Proposition \ref{prop_module_2} (1), we have 
\[
\sqrt{\Ann_{\Lambda(K_{\alpha}/k)} (X/S_{\alpha} X)} = \sqrt{P + (S_{\alpha})}/(S_{\alpha}) \supset ((1 + T_2)^{p^N}-1),
\]
where the right hand side is seen as an ideal of $\Z_p[[T_2]]$.
Therefore $(X_{\Gal(K^{(2)}/K_{\alpha})})^{\dagger} \sim_{\Lambda^{\dagger}(K_{\alpha}/k)} 0$ by Lemma \ref{lem_dagger}.
Similarly 
\[
\sqrt{\Ann_{\Lambda(K_{\alpha}/k)} (X[S_{\alpha}])} \supset \sqrt{P + (S_{\alpha})}/(S_{\alpha}) \supset ((1 + T_2)^{p^N}-1)
\]
implies that $(X^{\Gal(K^{(2)}/K_{\alpha})})^{\dagger} \sim_{\Lambda^{\dagger}(K_{\alpha}/k)} 0$.
This completes the proof of the lemma.
\end{proof}

\begin{lem}\label{lem_module_4}
Let $f \in \Lambda(\KKip/k)$ be a nonzero element.
Then the natural image $\overline{f} \in \Lambda(K^{(i)}/k)$ of $f$ is nonzero for all but finitely many $\Z_p^i$-extensions $K^{(i)} \subset \KKip$ of $k$.
\end{lem}

\begin{proof}
For $\alpha \in \Z_p^i$, define $K_{\alpha}$ and $S_{\alpha}$ as in the proof of Lemma \ref{lem_module}.
It is enough to show that $\overline{f} \in \Lambda(K_{\alpha}/k)$ is nonzero for all but finitely many $\alpha \in \Z_p^i$.
Since $\Lambda(\KKip/k)$ is a UFD, we can suppose that $f$ is a prime element.
Clearly $\overline{f} \in \Lambda(K_{\alpha}/k)$ is zero $\Leftrightarrow$ $f \in (S_{\alpha})$ $\Leftrightarrow$ $(f) = (S_{\alpha})$, which holds for at most one $\alpha$.
This proves the lemma.
\end{proof}

\begin{thm}\label{thm_module_2}
Let $i \geq 1$ and $K^{(i+1)}$ be a $\Z_p^{i+1}$-extension of $k$.
Let $X$ be a finitely generated torsion $\Lambda(K^{(i+1)}/k)$-module.
Suppose that
\[
X \sim \bigoplus_{l=1}^t \Lambda(K^{(i+1)}/k) / (f_l)
\]
where $f_l$ is a nonzero element of $\Lambda(K^{(i+1)}/k)$.
Then 
\[
(X_{\Gal(K^{(i+1)}/K^{(i)})})^{\dagger} \sim \bigoplus_{l=1}^t \Lambda^{\dagger}(K^{(i)}/k) / (\overline{f_l})
\]
for generic $\Z_p^i$-extensions $K^{(i)} \subset K^{(i+1)}$ of $k$.
\end{thm}

\begin{proof}
Take a pseudo-isomorphism $X \to \bigoplus_{l=1}^t \Lambda(K^{(i+1)}/k) / (f_l)$ and let $X', X''$, and $X'''$ be the kernel, image, and the cokernel of the map.
Then we have the short exact sequences
\[
0 \to X' \to X \to X'' \to 0
\]
and
\[
0 \to X'' \to \bigoplus_{l=1}^t \Lambda(K^{(i+1)}/k) / (f_l) \to X''' \to 0.
\]
For any $\Z_p^i$-extension $K^{(i)} \subset \KKip$ of k, these yield exact sequences
\[
(X')_{\Gal(K^{(i+1)}/K^{(i)})} \to X_{\Gal(K^{(i+1)}/K^{(i)})} \to (X'')_{\Gal(K^{(i+1)}/K^{(i)})} \to 0
\]
and
\[
(X''')^{\Gal(K^{(i+1)}/K^{(i)})} \to (X'')_{\Gal(K^{(i+1)}/K^{(i)})} 
\to \bigoplus_{l=1}^t \Lambda(K^{(i)}/k) / (\overline{f_l}) \to (X''')_{\Gal(K^{(i+1)}/K^{(i)})} \to 0.
\]
Then since $X'$ and $X'''$ are pseudo-null, Lemma \ref{lem_module} implies that for generic $\Z_i$-extensions $K^{(i)} \subset K^{(i+1)}$ of $k$, we have
\begin{eqnarray*}
(X_{\Gal(K^{(i+1)}/K^{(i)})})^{\dagger}
&\sim& ((X'')_{\Gal(K^{(i+1)}/K^{(i)})})^{\dagger} \\
&\sim& \left(\bigoplus_{l=1}^t \Lambda(K^{(i)}/k) / (\overline{f_l})\right)^{\dagger} \\
&\simeq& \bigoplus_{l=1}^t \Lambda^{\dagger}(K^{(i)}/k) / (\overline{f_l}),
\end{eqnarray*}
which proves the theorem.
\end{proof}

\begin{proof}[Proof of Theorem \ref{thm_descent_step}]
For a $\Z_p^i$-extension $K^{(i)} \subset K^{(i+1)}$ of $k$, we have
\[
X(K^{(i+1)})_{\Gal(K^{(i+1)}/K^{(i)})} = \Gal(\LL/K^{(i+1)}),
\]
where $\LL$ is the maximal abelian extension of $K^{(i)}$ contained in $L(K^{(i+1)})$.
We have a natural short exact sequence of $\Lambda(K^{(i)}/k)$-modules
\[
0 \to X(K^{(i+1)})_{\Gal(K^{(i+1)}/K^{(i)})} \to \Gal(\LL/K^{(i)}) \to \Gal(K^{(i+1)}/K^{(i)}) \to 0.
\]
By the definition of $\Lambda^{\dagger}(K^{(i)}/k)$, it can be seen that $\Gal(K^{(i+1)}/K^{(i)})^{\dagger} \sim 0$.
Therefore we have $\left(X(K^{(i+1)})_{\Gal(K^{(i+1)}/K^{(i)})}\right)^{\dagger} \sim \Gal(\LL/K^{(i)})^{\dagger}$, which implies by Lemma \ref{lem_module_4} and Theorem \ref{thm_module_2}
\[
\bigoplus_{l=1}^t \Lambda^{\dagger}(K^{(i)}/k) / (\overline{f_l}) \sim \Gal(\LL/K^{(i)})^{\dagger}
\]
for generic $\Z_p^i$-extensions $K^{(i)} \subset K^{(i+1)}$ of $k$.
Here we used the fact that the relation $\sim$ is an equivalence relation on
finitely generated torsion modules (see \cite[Remarks after Proposition (5.1.7)]{NSW}).

On the other hand, since $L(K^{(i)})$ is the maximal unramified extension of $K^{(i)}$ contained in $\LL$, we have a short exact sequence of $\Lambda(K^{(i)}/k)$-modules
\[
0 \to \sum_{\pe \in S_p(k)} \sum_{\Pe | \pe} I_{\Pe}(\LL/K^{(i)}) \to \Gal(\LL/K^{(i)}) \to X(K^{(i)}) \to 0,
\]
where $\Pe$ runs through the primes of $K^{(i)}$ above $\pe$ and $\sum$ means the generated closed subgroup.
In particular, we have a surjective homomorphism $\Gal(\LL/K^{(i)}) \twoheadrightarrow X(K^{(i)})$, which proves the assertion (1).

Next we prove the assertion (2).
We put $\II_{\pe}(\LL/\KKi) = \sum_{\Pe | \pe} I_{\Pe}(\LL/K^{(i)})$, which is a closed subgroup of $\Gal(\LL/\KKi)$.
By the above argument, it is enough to show that for every $\pe \in S_p(k)$ we have $\II_{\pe}(\LL/K^{(i)})^{\dagger} \sim 0$ for generic $\Z_p^i$-extensions $K^{(i)} \subset K^{(i+1)}$ of $k$.

Choose a prime $\Pe_0 | \pe$ of $K^{(i)}$.
Since $\LL/K^{(i+1)}$ is unramified, we have $I_{\Pe_0}(\LL/K^{(i)}) \simeq I_{\Pe_0}(K^{(i+1)}/K^{(i)})$.
Choose a topological generator $\rho$ of $I_{\Pe_0}(\LL/K^{(i)})$.
Consider the $\Z_p$-homomorphism $\Z_p[\Gal(K^{(i)}/k)] \to \II_{\pe}(\LL/K^{(i)})$ which sends $\sigma \in \Gal(K^{(i)}/k)$ to $\tilde{\sigma} \rho \tilde{\sigma}^{-1} \in I_{\sigma(\Pe_0)}(\LL/K^{(i)})$, where $\tilde{\sigma} \in \Gal(\LL/k)$ is a lift of $\sigma$.
It is clearly $\Z_p[\Gal(K^{(i)}/k)]$-homomorphism and the compactness of $\II_{\pe}(\LL/K^{(i)})$ implies that it extends to a surjective $\Lambda(K^{(i)}/k)$-homomorphism
\[
\Lambda(K^{(i)}/k) \twoheadrightarrow \II_{\pe}(\LL/K^{(i)}).
\]

If $\sigma \in D_{\pe}(K^{(i)}/k)$, then $\sigma(\Pe_0) = \Pe_0$, the injectivity of $I_{\Pe_0}(\LL/K^{(i)}) \to I_{\Pe_0}(K^{(i+1)}/K^{(i)})$, and the commutativity of $\Gal(K^{(i+1)}/k)$ imply that
$\tilde{\sigma} \rho \tilde{\sigma}^{-1} = \rho$.
In other words, $\sigma - 1$ is contained in the kernel of the above surjective homomorphism.

Suppose that $i \geq 2$.
Then the assumption that $\rank_{\Z_p} D_{\pe}(K^{(i+1)}/k) \geq 2$ implies that $\rank_{\Z_p} D_{\pe}(K^{(i)}/k) \geq 2$ for generic $\Z_p^i$-extensions $K^{(i)} \subset K^{(i+1)}$ of $k$ by Proposition \ref{prop_almost_decomp_2}.
For such $K^{(i)}$, the above argument shows that $\II_{\pe}(\LL/K^{(i)}) \sim 0$, as claimed.

Finally suppose that $i = 1$ and choose any $\Z_p$-extension $K \subset K^{(2)}$ of $k$.
Then the assumption that $\rank_{\Z_p} D_{\pe}(K^{(2)}/k) = 2$ implies that we can choose a non-identity element $\gamma \in D_{\pe}(K/k)$.
The above argument shows that there is a surjective homomorphism $\Lambda(K/k)/(\gamma-1)  \twoheadrightarrow \II_{\pe}(\LL/K)$,
which proves that $\II_{\pe}(\LL/K)^{\dagger} \sim 0$, as claimed.
This completes the proof of Theorem \ref{thm_descent_step}.
\end{proof}

\section{Proof of Theorem \ref{thm_characteristic}}\label{sec_characteristic}

\begin{lem}\label{lem_cft}
Let $S \subset S_p(k)$ and
$k'/k$ a finite cyclic extension in which no primes in $S$ split.
We denote by $S'$ the set of primes of $k'$ above a prime in $S$.
Let $\MM$ be the maximal extension of $k'$ contained in $M_{S'}(k')$ such that 
the natural action of $\Gal(k'/k)$ on the inertia $I_{\pe'}(\MM/k')$ is trivial for every $\pe' \in S'$.
(Note that $I_{\pe'}(M_{S'}(k')/k')$ is stable under the action of $\Gal(k'/k)$ since $\pe'$ does not split in $k'/k$.)
Then $\MM/M_S(k)$ is a finite extension.
\end{lem}

\begin{proof}

We mimic the calculation of \cite[Proposition 1]{fujii}.

The extension $\MM/M_S(k)$ is finite if and only if the kernel of the restriction map 
$\Gal(\MM/L(k')) \to \Gal(M_S(k)/L(k))$ is finite.
Let $\sigma$ be a generator of $\Gal(k'/k)$.
By Theorem \ref{thm_cft}, we have
\[
\Gal(M_S(k)/L(k)) \simeq \prod_{\pe \in S} U_{\pe}^{(1)} \bigg/ B_{k},
\]
where $B_{k}$ denotes the diagonal image of $E_{k} \otimes \Z_p$ in $\prod_{\pe \in S} U_{\pe}^{(1)}$.
On the other hand,
\[
\Gal(\MM/L(k')) \simeq \prod_{\pe' \in S'} U_{\pe'}^{(1)} \bigg/ B_{k'} \prod_{\pe' \in S'} (\sigma-1)U_{\pe'}^{(1)},
\]
where $B_{k'}$ denotes the image of $E_{k'} \otimes \Z_p$ in $\prod_{\pe' \in S'} U_{\pe'}^{(1)}$.
Consider the commutative diagram with exact rows
\[
\begin{CD}
0 @>>> \frac{B_{k'} \prod_{\pe' \in S'} (\sigma-1) U_{\pe'}^{(1)}}{\prod_{\pe' \in S'} (\sigma-1) U_{\pe'}^{(1)}}
@>>> \frac{\prod_{\pe' \in S'} U_{\pe'}^{(1)}}{\prod_{\pe' \in S'} (\sigma-1) U_{\pe'}^{(1)}}
@>>> \frac{\prod_{\pe' \in S'} U_{\pe'}^{(1)}}{B_{k'} \prod_{\pe' \in S'} (\sigma-1) U_{\pe'}^{(1)}}
@>>> 0 \\
@. @VVV @VVV @VVV @.\\
0 @>>> B_k @>>> \prod_{\pe \in S} U_{\pe}^{(1)} @>>> \frac{\prod_{\pe \in S} U_{\pe}^{(1)}}{B_k} @>>> 0,
\end{CD}
\]
where the vertical maps are induced by the norm map.
In order to show that the right vertical map has finite kernel,
we show that the kernel of the middle map and the cokernel of the left vertical map are finite.

The middle vertical map can be divided into each component
\[
U_{\pe'}^{(1)}\bigg/ (\sigma-1) U_{\pe'}^{(1)} \to U_{\pe}^{(1)},
\]
for $\pe \in S$ and $\pe' \in S'$ with $\pe' | \pe$.
This map has finite cokernel since the image contains $(U_{\pe}^{(1)})^{[k':k]}$.
On the other hand, as the left and the right hand side is the cokernel and the kernel of 
\[
\begin{CD}
U_{\pe'}^{(1)} @>{\sigma-1}>> U_{\pe'}^{(1)},
\end{CD}
\]
respectively, the $\Z_p$-ranks of them coincide.
Therefore the kernel is also finite, as claimed.
The finiteness of the left vertical map also follows from the fact that the image of $E_{k'}$ under the norm map $E_{k'} \to E_k$ contains
$(E_k)^{[k':k]}$.
This completes the proof of the lemma.
\end{proof}

\begin{prop}\label{prop_characteristic}
Let $K \in \EE_{\ns}(k)$ and choose a topological generator $\sigma$ of $\Gal(K/k)$.
Then $\ch_{\Lambda(K/k)} X(K)$ is prime to $(\sigma^{p^N}-1)/(\sigma-1)$ for all positive integer $N$.
\end{prop}

\begin{proof}
We shall show that the natural surjective map 
\[
X(K) / (\sigma^{p^N}-1)X(K) \to X(K) / (\sigma-1)X(K)
\]
is pseudo-isomorphic.
We have $X(K) / (\sigma-1)X(K) = \Gal(\LL_0/K)$, 
where $\LL_0$ is the maximal abelian extension of $k$ contained in $L(K)$.
Similarly, let $k_N$ be the $N$-th layer of the $\Z_p$-extension $K/k$, then
$X(K) / (\sigma^{p^N}-1)X(K) = \Gal(\LL_N/K)$, 
where $\LL_N$ is the maximal abelian extension of $k_N$ contained in $L(K)$.
It is clear that $\LL_0 \subset M_{S_p(k)}(k)$ and $\LL_N \subset M_{S_p(k_N)}(k_N)$.

For every prime $\pe_N \in S_p(k_N)$, since $\pe_N$ does not split in $k_N/k$ by $K \in \EE_{\ns}(k)$ and 
the inertia group $I_{\pe_N}(\LL_0/k_N)$ injects into $\Gal(K/k_N)$, the Galois group $\Gal(k_N/k)$ acts on $I_{\pe_N}(\LL_0/k_N)$ trivially.
Define $\MM$ similarly as in Lemma \ref{lem_cft}, namely,
let $\MM$ be the maximal extension of $k_N$ contained in $M_{S_p(k_N)}(k_N)$ such that 
the natural action of $\Gal(k_N/k)$ on the inertia subgroups $I_{\pe_N}(\MM/k_N)$ is trivial for every prime $\pe_N$ of $k_N$ above $p$.
Then the above argument shows that $\LL_N \subset \MM$.
Lemma \ref{lem_cft} shows that $\MM / M_{S_p(k)}(k)$ is a finite extension.

By definition, $\LL_0 = M_{S_p(k)}(k) \cap L(K)$ and $\LL_N = \MM \cap L(K)$, hence
$\LL_0 = \LL_N \cap M_{S_p(k)}(k)$.
Therefore the finiteness of $\MM / M_{S_p(k)}(k)$ implies the finiteness of $\LL_N/\LL_0$.
This completes the proof of the proposition.
\end{proof}

\begin{proof}[Proof of Theorem \ref{thm_characteristic}]
By Proposition \ref{prop_characteristic} and the assumption that $X(K)^{\dagger} \sim 0$, $\ch_{\Lambda(K/k)}(X(K))$ is a power of $(\sigma-1)$, where $\sigma$ is a topological generator of $\Gal(K/k)$.
Then by Lemma \ref{lem_jaulent} and the assumption that $K/k$ is arithmetically semi-simple, 
$\ch_{\Lambda(K/k)}X(K) = \ch_{\Lambda(K/k)}(X(K)/(\sigma-1)X(K)) = (\sigma-1)^{s(K/k)}$.
This completes the proof of Theorem \ref{thm_characteristic}.
\end{proof}


\section{Proof of Theorem \ref{thm_openness}}\label{sec_openness}
The following proposition is a generalization of \cite[Theorem 1]{fukuda}.
It is of independent interest.

\begin{prop}\label{prop_fukuda}
Let $K/k$ be a $\Z_p$-extension and $k_n$ be the $n$-th layer of it.
Take a non-negative integer $n_0$ such that $K/k_{n_0}$ is totally ramified at every ramified prime.
Then $X(K)$ is a finitely generated $\Z_p$-module of rank $s(K/k)$ if and only if
there is an integer $n \geq n_0$ such that $\sharp X(k_{n+1}) = p^{s(K/k)} \sharp X(k_n)$.
\end{prop}

\begin{proof}
The ``only if'' part follows immediately from Iwasawa's class number formula.
In order to show the ``if'' part, let $n \geq n_0$ be an integer in the statement.
Put $Y = \Ker(X(K) \twoheadrightarrow X(k_n))$, which is a sub $\Lambda(K/k)$-module of $X(K)$ of finite index $\sharp X(k_n)$.
Choose a topological generator $\sigma$ of $\Gal(K/k)$ and put $\nu_{n+1, n} = (\sigma^{p^{n+1}}-1) / (\sigma^{p^n}-1)$.
Then in the proof of Iwasawa's class number formula, it is shown that $X(k_{n+1}) = X/\nu_{n+1, n}Y$.
Therefore by the choice of $n$, we have
\[
[Y : \nu_{n+1, n}Y] = p^{s(K/k)}.
\]

On the other hand, since
\[
\rank_{\Z_p} Y / (\sigma - 1)Y = \rank_{\Z_p} X(K) / (\sigma - 1)X(K) = s(K/k),
\]
we have a surjective $\Lambda(K/k)$-homomorphism $Y \twoheadrightarrow \Z_p^{s(K/k)}$,
where $\Gal(K/k)$ acts on $\Z_p$ trivially.
Let $Z$ be the kernel of the map.
We have the commutative diagram with exact rows
\[
\begin{CD}
0 @>>> Z @>>> Y @>>> \Z_p^{s(K/k)} @>>> 0\\
@. @V {\nu_{n+1, n}} VV @V {\nu_{n+1, n}} VV @V {\nu_{n+1, n}} VV @.\\
0 @>>> Z @>>> Y @>>> \Z_p^{s(K/k)} @>>> 0.
\end{CD}
\]
Since $\nu_{n+1, n} = \sigma^{p^n (p-1)} + \dots + \sigma^{p-1} + 1$ acts on $\Z_p$ as the multiplication by $p$, the snake lemma yields
\[
0 \to Z/\nu_{n+1, n} Z \to Y / \nu_{n+1, n} Y \to (\Z/p\Z)^{s(K/k)} \to 0.
\]
Since the order of the middle term is $p^{s(K/k)}$, we have $Z/\nu_{n+1, n} Z = 0$.
Hence Nakayama's lemma implies that $Z = 0$ and therefore $Y \simeq \Z_p^{s(K/k)}$.
Consequently $X(K)$ is a finitely generated $\Z_p$-module of rank $s(K/k)$, as asserted.
\end{proof}

We remark that the proof of the ``if'' part implies the following:
If there is an integer $n \geq n_0$ such that $\sharp X(k_{n+1}) = p^s \sharp X(k_n)$ with $s \leq s(K/k)$,
then $s = s(K/k)$ and $X(K)$ is a finitely generated $\Z_p$-module of rank $s(K/k)$.

\begin{proof}[Proof of Theorem \ref{thm_openness}]
Choose any element $K_0$ in the concerned set and let $k_n$ be the $n$-th layer of 
the $\Z_p$-extension $K_0/k$.
Then $s(k) \leq s(K_0/k) \leq \rank_{\Z_p} X(K_0) = s(k)$ shows that $s(K_0/k) = s(k)$.
Since $K_0 \in \EE_{\ram}(k)$, there is a non-negative integer $n_0$ such that
any prime of $k_{n_0}$ above $p$ is totally ramified in $K_0/k_{n_0}$.
By Proposition \ref{prop_fukuda}, there is an integer $n \geq n_0$ such that $\sharp X(k_{n+1}) = p^{s(k)} \sharp X(k_n)$. 

Take any $K \in \EE(k)$ such that $[K \cap K_0:k] \geq p^{n+1}$.
Since the $n$-th layers and $(n+1)$-st layers of $K/k$ and $K_0/k$ coincide,
it is clear that $K \in \EE_{\ram}(k)$.
Moreover $\sharp X(k_{n+1}) = p^{s(k)} \sharp X(k_n)$ and the remark after Proposition \ref{prop_fukuda} imply that
$X(K)$ is a finitely generated $\Z_p$-module of rank $s(k) = s(K/k)$.
Consequently, $K$ is contained in the concerned set.
By Remark \ref{rem_top_greenberg}, this completes the proof of Theorem \ref{thm_openness}.
\end{proof}

\section{Proofs of Lemmas \ref{lem_rank} and \ref{lem_induction}}\label{sec_appendix}

\begin{lem}\label{lem_polynomial}
Let $d$ be a positive integer and $f(T) = f(T_1, \dots, T_d) \in \Z_p[T_1, \dots, T_d]$ a nonzero polynomial.
Then $f(\alpha) \neq 0$ for generic $\alpha \in \Z_p^d$ with respect to the natural (Haar) measure of $\Z_p^d$.
\end{lem}

\begin{proof}
Put $E = \{\alpha \in \Z_p^d \mid f(\alpha) = 0\}$.
Since $E$ is a closed subset of $\Z_p^d$, it is enough to show that the measure of $E$ is zero.
We prove it by induction on $d$.
If $d = 1$, then $E$ is a finite set and the measure is zero, as claimed.

Suppose $d \geq 2$.
Let $\mu', \nu$, and $\mu = \mu' \otimes \nu$ be the measures of $\Z_p^{d-1}, \Z_p$, and $\Z_p^d$, respectively.
If $f$ is a constant polynomial, then the statement is trivial.
Otherwise there is an indeterminate which appears in $f$, 
so without loss of generality, we suppose that $T_d$ appears in $f$.
We write $T' = (T_1, \dots, T_{d-1})$ for short.
Then we can write 
\[
f(T) = \sum_{k=0}^N g_k(T') T_d^k
\]
for some positive integer $N$ and polynomials $g_k(T')$ with $g_N \neq 0$.
By the induction hypothesis, $E' = \{\alpha' \in \Z_p^{d-1} \mid g_N(\alpha') = 0\}$ satisfies $\mu'(E') = 0$.
Moreover, if $\alpha' \in \Z_p^{d-1} \setminus E'$, then the set 
$ E_{\alpha'} = \{ \alpha_d \in \Z_p \mid f(\alpha', \alpha_d)=0 \}$ is finite and in particular $\nu(E_{\alpha'}) = 0$.

Putting all together,
\[
\mu(E) 
= \int_{\Z_p^{d-1}} \nu(E_{\alpha'}) d\mu'(\alpha')
= \int_{\Z_p^{d-1} \setminus E'} \nu(E_{\alpha'}) d\mu'(\alpha') + \int_{E'} \nu(E_{\alpha'}) d\mu'(\alpha')
 = 0.
\]
(See for example \cite[section 35]{halmos}.)
This completes the proof.
\end{proof}

For convenience, we introduce the following terminology:
Let $X$ be a topological space equipped with a Borel measure and $A$ a subset of $X$.
We say that $A \subset X$ is generic (resp. large, resp. weakly large) if generic (resp. almost all, resp. weakly almost all) $x \in X$ is an element of $A$.

\begin{proof}[Proof of Lemma \ref{lem_rank}]
(1) Although it is not difficult to prove the second assertion simultaneously with the first assertion, we deduce the second from the first here.
We can suppose that $\rank_{\Z_p} L_j = i$.
Choose a submodule $L_j'$ of $M$ containing $L_j$ such that $\rank_{\Z_p} L_j' = i'$.
Then by the first assertion, for generic $N \in \Gr(i', M)$, we have $\rank_{\Z_p} (\Image (L_j' \to M/N)) = i'$ and consequently $\rank_{\Z_p} (\Image (L_j \to M/N)) = i$, as claimed.

Now we shall prove the first assertion.
It is clear that $\rank_{\Z_p} (\Image(L_j \to M/N))=i$ if and only if $N + L_j$ has finite index in $M$.
Choose any basis of $M$ and identify $M$ with $\Z_p^d$ (the module of column vectors).
The map of Lemma \ref{lem_measure_calculation} is read as
\[
\begin{array}{ccccc}
M_{i,d-i}(\Z_p) &\hookrightarrow& \GL_d(\Z_p) &\twoheadrightarrow& \Gr(i, \Z_p^d).\\
\alpha & \mapsto &
\begin{pmatrix}
1_{d-i} & 0 \\ \alpha & 1_i 
\end{pmatrix}
& \mapsto& N_{\alpha} =
\left \{
\begin{pmatrix}1_{d-i} & 0 \\ \alpha & 1_i \end{pmatrix}
\begin{pmatrix} x \\ 0\end{pmatrix}
\middle| x \in \Z_p^{d-i} \right \}
\end{array}
\]
By Lemmas \ref{lem_measure_calculation} and \ref{lem_covering}, it is enough to show that $N_{\alpha} + L_j$ has finite index in $M$ for $1 \leq j \leq r$ for generic $\alpha \in M_{i,d-i}(\Z_p)$.

Choose $\Z_p$-linear independent elements $b_1^{(j)}, \dots, b_i^{(j)}$ of $L_j$.
Put
\[
f_j(T) = \det \begin{pmatrix}
\begin{matrix} 1_{d-i} \\ T \end{matrix}
& b_1^{(j)} & \cdots & b_i^{(j)}  
\end{pmatrix}
\]
where $T = \begin{pmatrix}
T_{1,1} & \cdots & T_{1, d-i} \\
\vdots & \ddots & \vdots \\
T_{i, 1} & \cdots & T_{i, d-i}
\end{pmatrix}
$
is a tuple of indeterminates.
Then for $\alpha \in M_{i, d-i}(\Z_p)$, $N_{\alpha} + L_j \subset M$ has finite index if and only if $f(\alpha)$
does not vanish.
Clearly $f_j(T)$ is a polynomial of $i(d-i)$ variables with coefficients in $\Z_p$.
Moreover, using the linear independence of $b_1^{(j)}, \dots, b_i^{(j)}$, one can check that there exists an element $\alpha \in M_{i, d-i}(\Z_p)$ such that $f_j(\alpha) \neq 0$.
Therefore $f_j(T)$ is not zero as a polynomial.

Now put $f(T) = f_1(T)\dots f_r(T)$, which is a nonzero polynomial.
Then $\rank_{\Z_p} (\Image(L_j \to M/N_{\alpha}))=i$ for all $1 \leq j \leq r$ if and only if $f(\alpha) \neq 0$.
By Lemma \ref{lem_polynomial}, $f(\alpha) \neq 0$ for generic $\alpha \in M_{i,d-i}(\Z_p)$.
This proves (1).

(2) Put 
\[
\FF = \{ N \in \Gr(1, M) \mid \text{$\rank_{\Z_p}(\Image(L_j \to M/N)) = 1$ for all $1 \leq j \leq r$ and $s(N) = s$} \}.
\]
Then our aim is to prove that $\FF \subset \Gr(1, M)$ is generic.

As in (1), choose any basis of $M$ and identify $M$ with $\Z_p^d$.
Let $e_1, \dots, e_d \in \Z_p^d$ be the standard basis.
If $\alpha = (\alpha_1, \dots, \alpha_{d-1}) \in M_{1,d-1}(\Z_p)$, then 
\[
N_{\alpha} = \langle e_1 + \alpha_1 e_d, \dots, e_{d-1} + \alpha_{d-1} e_d \rangle = 
\left\{ \sum_{k=1}^d x_k e_k \in M \middle| x_k \in \Z_p, \sum_{k=1}^d \alpha_i x_k = 0 \right\},
\]
setting $\alpha_d = -1$ for convenience.
By Lemma \ref{lem_measure_calculation}, $U = \{ N_{\alpha} \in \Gr(1, \Z_p^d) \mid \alpha \in M_{1,d-1}(\Z_p) \}$ is an open set of $\Gr(1, \Z_p^d)$.
By Lemma \ref{lem_covering}, we have an open covering $\Gr(1, M) = \bigcup U_W$ consisting of the similarly constructed open sets, and it is easy to see that the open sets intersect each other.
In fact, each such open sets contains
\[
\left\{ \sum_{k=1}^d x_k e_k \in \Z_p^d \middle| x_k \in \Z_p, \sum_{k=1}^d x_k = 0 \right\} \in \Gr(1, \Z_p^d),
\]
for example.

\begin{claim}\label{claim_rank}
If $\FF \cap U_W \neq \emptyset$, then $\FF \cap U_W \subset U_W$ is generic.
\end{claim}

Let us deduce the assertion (2) from Claim \ref{claim_rank} in advance.
Since $\FF \neq \emptyset$ by the definition of $s$,
choose $W$ such that $\FF \cap U_W \neq \emptyset$.
Then by Claim \ref{claim_rank} applied to $U_W$, $\FF \cap U_W \subset U_W$ is generic.
Next for any other $W'$, $\FF \cap U_{W'} \neq \emptyset$ since $U_{W'} \cap U_W$ is a non-empty open subset of $U_W$.
Applying Claim \ref{claim_rank} again to $U_{W'}$, we have $\FF \cap U_{W'} \subset U_{W'}$ is generic.
Consequently $\FF \subset \Gr(1, \Z_p^d)$ is generic, which proves (2).

\begin{proof}[Proof of Claim \ref{claim_rank}]
It is enough to prove the claim for $U_W = U$.
By (1), we may work only for $\alpha \in M_{1,d-1}(\Z_p)$ such that $\rank_{\Z_p} (\Image(L_j \to M/N_{\alpha}))=1$ (i.e., $L_j \not\subset N_{\alpha}$) for all $1 \leq j \leq r$.

Take a $\Z_p$-basis $\left\{\sum_{k=1}^d b_k^{(j,\nu)} e_k \right\}_{\nu}$ of $L_j$, where $\nu$ runs through a set with $\rank_{\Z_p} L_j$ elements.
Since $L_j \not \subset N_{\alpha}$, we have 
$\sum_{k=1}^d  b_k^{(j,\nu_0)} \alpha_k \neq 0$ for some $\nu_0$.
Then for $\nu \neq \nu_0$,
\[
\left(\sum_{l=1}^d b_l^{(j,\nu_0)} \alpha_l\right) \sum_{k=1}^d b_k^{(j,\nu)} e_k
- \left(\sum_{l=1}^d b_l^{(j,\nu)} \alpha_l \right) \sum_{k=1}^d b_k^{(j,\nu_0)} e_k
 = \sum_{k=1}^d \left( \sum_{l=1}^d \left(b_k^{(j,\nu)} b_l^{(j,\nu_0)} - b_l^{(j,\nu_0)} b_k^{(k,\nu)}\right) \alpha_l \right) e_k
\]
is contained in $N_{\alpha} \cap L_j$.
Moreover they form a basis of a submodule of $N_{\alpha} \cap L_j$ of finite index
because the linear independence is clear and the injective map
\[
L_j / (N_{\alpha} \cap L_j) \hookrightarrow M/N_{\alpha} \simeq \Z_p
\]
shows that $\rank_{\Z_p} (N_{\alpha} \cap L_j) = \rank_{\Z_p} L_j - 1$.
This shows that $\sum_{j=1}^r (N_{\alpha} \cap L_j)$ has a submodule of finite index generated by
\[
\sum_{k=1}^d \left( \sum_{l=1}^d (b_k^{(j,\nu)} b_l^{(j,\nu')} - b_l^{(j,\nu')} b_k^{(j,\nu)}) \alpha_l \right) e_k
\]
where $j, \nu, \nu'$ run arbitrarily (the range of $\nu$ and $\nu'$ depends on $j$).
Therefore 
\[
s(N_{\alpha}) = d - 1 - \rank \left( \sum_{l=1}^d (b_k^{(j,\nu)} b_l^{(j,\nu')} - b_l^{(j,\nu')} b_k^{(j,\nu)}) \alpha_l \right),
\]
where on the right hand side the rank means the rank as 
a matrix whose rows and columns are indexed by $1 \leq k \leq d$ and $(j,\nu,\nu')$, respectively.

Since $\FF \cap U \neq \emptyset$, there exists $\alpha'$ such that $s(N_{\alpha'}) = s$.
Then there exists a minor square matrix of size $d-1 - s$ whose determinant does not vanish for the $\alpha'$.
Since the determinant is a polynomial of $\alpha$ with coefficients in $\Z_p$, 
using Lemma \ref{lem_polynomial},
we conclude that it does not vanish and consequently $s(N_{\alpha})=s$ for generic $\alpha$.
This completes the proof of Claim \ref{claim_rank}.
\end{proof}
As already remarked, this completes the proof of Lemma \ref{lem_rank}.
\end{proof}

\begin{proof}[Proof of Lemma \ref{lem_induction}]
We prove the corresponding statement for a free $\Z_p$-module $M$ of rank $d$ and 
properties of free quotients of $M$.
Namely, let $P$ (resp. $Q$) be a property of free quotients of $M$ of rank $d'$ (resp. $d''$) and suppose
\begin{enumerate}
\item[$(a)$] $P(M/N)$ for weakly almost all $N \in \Gr(d',M)$, and
\item[$(b)$] for any $N \in \Gr(d',M)$, $P(M/N)$ implies $Q((M/N)/L)$ for weakly almost all $L \in \Gr(d'', M/N)$.
\end{enumerate}
Then we prove that $Q(M/N)$ for weakly almost all $N \in \Gr(d'', M)$.

Choose a basis of $M$ and identify $M$ with $\Z_p^d$.
Consider the map
\[
\begin{array}{cccccc}
\varphi: &M_{d'',d-d''}(\Z_p) &\hookrightarrow& \GL_d(\Z_p) &\twoheadrightarrow& \Gr(d'',M).\\
 &A & \mapsto & 
\begin{pmatrix} 1_{d-d''} & 0 \\ A & 1_{d''} \end{pmatrix}
& \mapsto & 
\left \{
\begin{pmatrix}1_{d-d''} & 0 \\ A & 1_{d''} \end{pmatrix}
\begin{pmatrix} x \\ 0\end{pmatrix}
\middle| x \in \Z_p^{d-d''} \right \}
\end{array}
\]
By Lemma \ref{lem_covering}, $\Gr(d'', M)$ is covered by the open set $\Image(\varphi)$ and the similar open sets.
Therefore it is enough to show that $Q(M/N)$ for weakly almost all $N \in \Image(\varphi)$.
Then by Lemma \ref{lem_measure_calculation}, it is enough to show that $Q(M/\varphi(A))$ for weakly almost all $A \in M_{d'',d-d''}(\Z_p)$.

Applying Lemma \ref{lem_measure_calculation} to the map
\[
\begin{array}{cccccc}
\psi:&M_{d',d-d'}(\Z_p) &\hookrightarrow& \GL_d(\Z_p) &\twoheadrightarrow& \Gr(d',M),\\
& B & \mapsto & 
\begin{pmatrix} 1_{d-d'} & 0 \\ B & 1_{d'} \end{pmatrix}
& \mapsto & 
\left \{
\begin{pmatrix}1_{d-d'} & 0 \\ B & 1_{d'} \end{pmatrix}
\begin{pmatrix} y \\ 0\end{pmatrix}
\middle| y \in \Z_p^{d-d'} \right \}
\end{array}
\]
the assumption (a) implies that the subset
\[
\BB = \{ B \in M_{d',d-d'}(\Z_p) \mid P(M/\psi(B)) \} \subset M_{d', d-d'}(\Z_p)
\]
is weakly large.
Next, for any $B \in M_{d', d-d'}(\Z_p)$, 
we choose as a basis of $M/\psi(B)$ the projection image of the last $d'$ elements of the fixed basis of $M$.
Then applying Lemma \ref{lem_measure_calculation} to the map
\[
\begin{array}{cccccc}
\psi_B: & M_{d'',d'-d''}(\Z_p) &\hookrightarrow & \GL_{d'}(\Z_p) &\twoheadrightarrow& \Gr(d'',M/\psi(B)),\\
&C & \mapsto & 
 \begin{pmatrix} 1_{d'} & 0 \\ C & 1_{d'-d''} \end{pmatrix}
& \mapsto & 
\left \{
\begin{pmatrix}1_{d'} & 0 \\ C & 1_{d'-d''} \end{pmatrix}
\begin{pmatrix} z \\ 0\end{pmatrix}
\middle| z \in \Z_p^{d'-d''} \right \}
\end{array}
\]
the assumption (b) implies that the subset
\[
\CC_B = \{ C \in M_{d'',d'-d''}(\Z_p) \mid Q((M/\psi(B))/\psi_B(C)) \} \subset M_{d'', d'-d''}(\Z_p)
\]
is weakly large if $P(M/\psi(B))$.

In the following proof, for each $B \in M_{d', d-d'}(\Z_p)$, 
let $B_1 \in M_{d'-d'', d-d'}(\Z_p)$ and $B_2 \in M_{d'', d-d'}(\Z_p)$ be the matrices such that $B = \begin{pmatrix} B_1 \\ B_2 \end{pmatrix}$.
Then under the natural inclusion map $\Gr(d'',M/\psi(B)) \hookrightarrow \Gr(d'', M)$, $\psi_B(C)$ is mapped to
\begin{eqnarray*}
&&\left\{
\begin{pmatrix} 1_{d-d'} & 0 &0\\ B_1 & 1_{d'-d''} & 0 \\ B_2 & C & 1_{d''}\end{pmatrix} \begin{pmatrix} y \\ z \\ 0 \end{pmatrix}
\middle| y \in \Z_p^{d-d'}, z \in \Z_p^{d'-d''} \right\}\\
&=& \left\{
\begin{pmatrix} 1_{d-d'} & 0 &0\\ 0 & 1_{d'-d''} & 0 \\ B_2 -C B_1 & C & 1_{d''} \end{pmatrix} \begin{pmatrix} x \\ 0 \end{pmatrix} \middle| x \in \Z_p^{d-d''} \right\}
=\varphi(B_2-CB_1, C),
\end{eqnarray*}
where $(B_2-CB_1, C)$ denotes a matrix in $M_{d'', d-d''}(\Z_p)$.

As a consequence, if $B \in \BB$ and $C \in \CC_B$, then $Q(M/\varphi(B_2-CB_1, C))$.
Hence it is enough to show that the subset
\[
\{ (B_2-CB_1, C) \in M_{d'',d-d''}(\Z_p) \mid B \in \BB, C \in \CC_B \} \subset M_{d'',d-d''}(\Z_p)
\]
is weakly large, under the assumptions that $\BB \subset M_{d', d-d'}(\Z_p)$ is weakly large and 
$\CC_B \subset M_{d'', d'-d''}(\Z_p)$ is weakly large if $B \in \BB$.
We prepare three claims in order to prove this.

\begin{claim}\label{claim_induction_1}
The subset 
\[
\{ (B, C) \in M_{d',d-d'}(\Z_p) \times M_{d'',d'-d''}(\Z_p) \mid B \in \BB, C \in \CC_B \} \subset M_{d',d-d'}(\Z_p) \times M_{d'',d'-d''}(\Z_p)
\]
is weakly large.
\end{claim}

\begin{proof}
Let $\mu_1, \mu_2$, and $\mu = \mu_1 \otimes \mu_2$ be the measures of $M_{d',d-d'}(\Z_p), M_{d'',d'-d''}(\Z_p)$, and $M_{d',d-d'}(\Z_p) \times M_{d'',d'-d''}(\Z_p)$, respectively.
Let $E$ be any measurable subset contained in 
\[
\left(M_{d',d-d'}(\Z_p) \times M_{d'',d'-d''}(\Z_p)\right) \setminus \{ (B, C) \in M_{d',d-d'}(\Z_p) \times M_{d'',d'-d''}(\Z_p) \mid B \in \BB, C \in \CC_B \}. 
\]
For any $B \in M_{d',d-d'}(\Z_p)$, put $E_B = \{ C \in M_{d'',d'-d''}(\Z_p) \mid (B,C) \in E\}$.
Then $E_B$ is measurable, the function $B \mapsto \mu_2(E_B)$ is measurable,
and 
\[
\mu(E) = \int_{M_{d',d-d'}(\Z_p)} \mu_2(E_B) d\mu_1(B)
\]
(see \cite[section 35]{halmos}).

Put $\BB' = \{ B \in M_{d', d-d'}(\Z_p) \mid \mu_2(E_B) = 0\}$.
The measurability of $B \mapsto \mu_2(E_B)$ implies that $\BB'$ is measurable.
Moreover $\BB' \supset \BB$.
In fact, if $B \in \BB$, then $\mu_2(E_B) = 0$ since $E_B \subset M_{d'',d'-d''}(\Z_p) \setminus \CC_B$ is measurable and $\CC_B \subset M_{d'', d'-d''}(\Z_p)$ is weakly large.
Therefore $\mu_1(M_{d',d-d'}(\Z_p) \setminus \BB') = 0$ since $\BB \subset M_{d',d-d'}(\Z_p)$ is weakly large.
Consequently
\[
\mu(E) = \int_{M_{d',d-d'}(\Z_p) \setminus \BB'} \mu_2(E_B) d\mu_1(B) + 
\int_{\BB'} \mu_2(E_B) d\mu_1(B) = 0.
\]
This completes the proof of the claim.
\end{proof}

\begin{claim}\label{claim_induction_2}
The map
\[
\begin{array}{cccc}
\theta: & M_{d',d-d'}(\Z_p) \times M_{d'',d'-d''}(\Z_p) &\to & M_{d',d-d'}(\Z_p) \times M_{d'',d'-d''}(\Z_p) \\
& \left( \begin{pmatrix} B_1 \\ B_2 \end{pmatrix}, C \right) & \mapsto & 
\left( \begin{pmatrix} B_1 \\ B_2-CB_1 \end{pmatrix}, C \right)
\end{array}
\]
is a homeomorphism preserving the measure.
\end{claim}

\begin{proof}
The map $\theta$ is a homeomorphism since the map
\[
\left( \begin{pmatrix} B_1 \\ B_2 \end{pmatrix}, C \right) \mapsto
\left( \begin{pmatrix} B_1 \\ B_2+CB_1 \end{pmatrix}, C \right)
\]
is the inverse of $\theta$.

For $C \in M_{d'', d'-d''}(\Z_p)$, let $\theta_C: M_{d',d-d'}(\Z_p) \to  M_{d',d-d'}(\Z_p)$ be the map
\[
\theta_C \begin{pmatrix} B_1 \\ B_2 \end{pmatrix} = 
\begin{pmatrix} B_1 \\ B_2-CB_1 \end{pmatrix},
\]
in other words, $\theta(B, C) = (\theta_C(B), C)$.
Then $\theta_C$ is a $\Z_p$-isomorphism and in particular preserves the measure.
Now take any measurable subset $E$ of $M_{d',d-d'}(\Z_p) \times M_{d'',d'-d''}(\Z_p)$ and put
\begin{eqnarray*}
E_C &=& \{ B \in M_{d',d-d'}(\Z_p) \mid (B, C) \in E \} \\
\theta(E)_C &=& \{ B \in M_{d',d-d'}(\Z_p) \mid (B, C) \in \theta(E) \}.
\end{eqnarray*}
Let $\mu_1, \mu_2$, and $\mu = \mu_1 \otimes \mu_2$ be the measures of 
$M_{d',d-d'}(\Z_p), M_{d'',d'-d''}(\Z_p)$, and $M_{d',d-d'}(\Z_p) \times M_{d'',d'-d''}(\Z_p)$, respectively.
Since $\theta(E)_C = \theta_C(E_C)$, we have $\mu_1(\theta(E)_C) = \mu_1(E_C)$.
Then
\[
\mu(\theta(E)) = \int_{M_{d'',d'-d''}(\Z_p)} \mu_1(\theta(E)_C) d\mu_2(C)
 = \int_{M_{d'',d'-d''}(\Z_p)} \mu_1(E_C) d\mu_2(C)
 = \mu(E),
\]
which completes the proof of the claim.
\end{proof}

\begin{claim}\label{claim_induction_3}
Let $X_1$ and $X_2$ be free $\Z_p$-modules of finite rank and put $X = X_1 \times X_2$.
We equip the natural measures on them.
If $A \subset X$ is weakly large,
then the image $\varpi(A) \subset X_1$ of $A$ under the projection $\varpi: X \to X_1$ is also weakly large.
\end{claim}
\begin{proof}
Let $\mu_1, \mu_2$, and $\mu = \mu_1 \otimes \mu_2$ be the measures of $X_1, X_2$, and $X$, respectively.
Let $E_1$ be any measurable subset of $X_1 \setminus \varpi(A)$.
Then $\varpi^{-1}(E_1)$ is a measurable subset of $X \setminus A$ and therefore $\mu(\varpi^{-1}(E_1)) = 0$ since $A \subset X$ is weakly large.
It is clear that, for any $x_2 \in X_2$, 
\[
\{ x_1 \in X_1 \mid (x_1, x_2) \in \varpi^{-1}(E_1) \} = E_1.
\]
Therefore 
\[
0 = \mu(\varpi^{-1}(E_1)) = \int_{X_2} \mu_1(E_1) d\mu_2 = \mu_1(E_1) \mu_2(X_2).
\]
Since $\mu_2(X_2)$ is nonzero, we have $\mu_1(E_1) = 0$, as claimed.
\end{proof}

Now we finish the proof of Lemma \ref{lem_induction}.
What we need to show is that 
\[
\{ (B_2-CB_1, C) \in M_{d'',d-d''}(\Z_p) \mid B \in \BB, C \in \CC_B \} \subset M_{d'',d-d''}(\Z_p)
\]
is weakly large.
But this set is the image of
\[
\{ (B, C) \in M_{d',d-d'}(\Z_p) \times M_{d'',d'-d''}(\Z_p) \mid B \in \BB, C \in \CC_B \},
\]
which is weakly large by Claim \ref{claim_induction_1}, under the composition of the maps
\[
\begin{array}{ccccc}
M_{d',d-d'}(\Z_p) \times M_{d'',d'-d''}(\Z_p) &\overset{\theta}{\to}& M_{d',d-d'}(\Z_p) \times M_{d'',d'-d''}(\Z_p)
& \to & M_{d'',d-d''}(\Z_p),\\
&&\left(\begin{pmatrix} B_1 \\ B_2 \end{pmatrix}, C \right) & \mapsto & (B_2, C)
\end{array}
\]
which preserves the weak largeness by Claims \ref{claim_induction_2} and \ref{claim_induction_3}.
This completes the proof of Lemma \ref{lem_induction}.
\end{proof}

\begin{rem}\label{rem_weakly}
The reason why we introduced the notion ``weakly almost all'' is to ensure Lemma \ref{lem_induction}.
Indeed, Claims \ref{claim_induction_1} and \ref{claim_induction_3} may fail if we omit the term ``weakly.''
The troubles lie in the possible failure of the measurabilities of the concerned set of Claim \ref{claim_induction_1} and the set $\varpi(A)$ of Claim \ref{claim_induction_3}.
\end{rem}

{\small
\bibliographystyle{alpha}
\bibliography{ggc_ref}
}

\end{document}